\pdfoutput=1
\documentclass{scrartcl}

\usepackage{imakeidx}
\makeindex[name=general]

\usepackage[]{algorithm2e}
\usepackage{bbm} 
\usepackage{dsfont} 
\usepackage{faktor}
\usepackage{tikz}
\usepackage{commath} 

\usepackage{amsmath, amsthm, amssymb}
\usepackage{mathabx} 
\usepackage{mathrsfs}
\usepackage[hyphens]{url}
\usepackage{amssymb,amsfonts}
\usepackage{color}
\usepackage{shuffle}

\usepackage[linecolor=white,backgroundcolor=white,bordercolor=white,textsize=tiny]{todonotes}
\let\todon\todo
\renewcommand\todo[1]{\todon{\color{red}#1}}
\usepackage{breqn}

\usepackage{silence}
\PassOptionsToPackage{hyphens}{url}
\usepackage[colorlinks,backref]{hyperref}
\usepackage{cleveref}
\usepackage{graphicx}
\usepackage[font={small,it}]{caption}
\usepackage[flushleft]{threeparttable}
\usepackage{longtable}
\usepackage{enumitem}

\allowdisplaybreaks

\newtheorem{theorem}{Theorem}[section]
\theoremstyle{plain}

\newtheorem{corollary}[theorem]{Corollary}

\newtheorem{example}[theorem]{Example}

\newtheorem{lemma}[theorem]{Lemma}

\newtheorem{remark}[theorem]{Remark}

\theoremstyle{definition} 

\newtheorem{definition}[theorem]{Definition}
\newtheorem*{tata}{Generalization}
  {\begin{mdframed}[backgroundcolor=lightgray]\begin{tata}}%
  {\end{tata}\end{mdframed}}

\usepackage{stmaryrd}

\usepackage{cite}
\usepackage{mleftright}
\mleftright

\title{Free generators and Hoffman's isomorphism for the two-parameter shuffle algebra}
\author{Leonard Schmitz\thanks{University of Greifswald} \and Nikolas Tapia\thanks{Weierstrass Institute}}
\date{\today}

\newcommand{\N}{\mathbb{N}}

\newcommand{\Q}{\mathbb{Q}}

\newcommand\diag{\operatorname{diag}}

\newcommand{\qShuffle}{\stackrel{\scalebox{0.6}{$\mathsf{qs}$}}{\shuffle}}

\newcommand{\blockShuffle}{\mathbin{\overset{\scalebox{0.8}{$\,\square$}}{\shuffle}}}

\newcommand\composition{\operatorname{Comp}}
\newcommand\groundRing{\mathbb{K}}
\newcommand\monoidComp{\mathfrak{M}}

\newcommand\ec{\mathsf{e}} 

\newcommand\sh{\operatorname{sh}}
\newcommand\qSh{\operatorname{qsh}}
\newcommand\SH{\operatorname{SH}}
\newcommand\QSH{\operatorname{QSH}}
\newcommand\id{\operatorname{id}}

\renewcommand\SS{\operatorname{SS}}
\newcommand\End{\operatorname{End}}
\newcommand\Sym{\operatorname{Sym}}

\newcommand\vectorize{\operatorname{vec}}

\newcommand\lm{\operatorname{lm}}

\newcommand\evaluate{\operatorname{eval}}

\newcommand\e[1]{{e}_{#1}} 
\newcommand\w[1]{{\color{cyan}\mathbf{#1}}}
\newcommand\DEF[1]{\textbf{#1}\index[general]{#1}}

\newcommand\compositionConnected{\composition_{\mathsf{con}}}

\newcommand\rows{\mathsf{rows}}
\newcommand\cols{\mathsf{cols}}
\newcommand\size{\mathsf{size}}

\newcommand\brew{\mathsf{fuse}}
\newcommand\bSH{\mathsf{bSH}}

\newcommand{\Ima}{\operatorname{Im}}

\newcommand\len{\operatorname{len}}
\newcommand\lex{\mathsf{lex}}
\newcommand\grlex{\mathsf{grlex}}

\DeclareMathOperator{\Log}{Log}
\DeclareMathOperator{\Exp}{Exp}

\begin{document}

\maketitle

\abstract{Signature transforms have recently been extended to data indexed by two and more parameters. 
With free Lyndon generators, ideas from $\mathbf{B}_\infty$-algebras and a novel two-parameter Hoffman exponential, we provide three classes of isomorphisms between the underlying two-parameter shuffle and quasi-shuffle algebras. In particular, we provide a Hopf algebraic connection to the (classical, one-parameter) shuffle algebra over the extended alphabet of connected matrix compositions.}

\tableofcontents

\section{Introduction}
The collection of iterated integrals of a path (termed \emph{path signature}) has been used as a tool for classification in various settings since Chen's seminal 1954 work \cite{chen54}.
In recent years---and in particular since pioneering work by Lyons and collaborators on the use of signatures as features for data streams in the early 2010s \cite{Gyu13,handw13}---there has been a surge in the number of applications, which range from calibration of financial models \cite{sigsde} to motion recognition \cite{Yang2022,CLT2019} and deep learning \cite{deepsig}.
We refer the interested reader to the surveys by Chevyrev and Kormilitzin \cite{ChevK16}, and Lyons and McLeod \cite{LyoMcLeod22}.
A discrete counterpart of the path signature has been introduced by one of the coauthors and collaborators \cite{dielh2020sig,diehl2020tropical}, and is likewise starting to see some applications \cite{diehl2023fruits}.

Signature transforms, both continuous and discrete, have been extended to data indexed by two and more parameters by either following a topological approach \cite{GLNO2022}, or an algebraic approach \cite{DS23}.
The algebraic picture in this setting becomes considerably more involved, and two-parameter analogues of the shuffle \cite{GLNO2022} and quasi-shuffle algebras \cite{DS23} enter the frame.

It is a well-known fact that the set of Lyndon words provides an algebraic independent generating set of the (classical, one-parameter) shuffle and quasi-shuffle algebra, e.g. \cite{RADFORD1979432}, \cite[Thm.~2.6]{H99} or \cite[Prop.~6.4.14]{grinberg2020hopf}.
The Hopf algebra structure from \cite[Thm.~2.12]{DS23} guarantees that the two-parameter quasi-shuffle algebra of matrix compositions is free. 
In this article, we address the following two questions posed in \cite[Sec.~6]{DS23}:
\begin{enumerate}
    \item \emph{Is there a free generating set akin to Lyndon words?}
    \item \emph{How are two-parameter shuffles and quasi-shuffles related to each other?}
\end{enumerate}

We provide full answers to both questions:  in  \Cref{thm:iso} we construct an explicit family ${(\mathcal{B}_w)}_{w\in\mathfrak{L}}$ of free generators indexed by Lyndon words, and in \Cref{thm_iso_sh_qsh} we introduce a two-parameter version of Hoffman's isomorphism together with its explicit inverse, verifying that the two-parameter shuffle and quasi-shuffle are isomorphic products in the matrix composition Hopf algebra. 

Note that  ${(\mathcal{B}_w)}_{w\in\mathfrak{L}}$ simultaneously generates not only the two-parameter shuffle and quasi-shuffle algebra, but even the (classical, one-parameter) shuffle algebra over the extended alphabet of connected matrix compositions. 
With this we relate the (classical,  one-parameter) framework \cite{H99} to the recent two-parameter development \cite{GLNO2022,DS23,ZLT22}. 
Using ideas from the theory of \(\mathbf{B}_\infty\)-algebras,  we can extend this connection even to the level of Hopf algebras. 

In \Cref{subsec:applic} we conclude with a fruitful application of our suggested machinery: a novel and transparent proof of the bialgebra relation with respect to diagonal deconcatenation and the two-parameter quasi-shuffle product, first stated  in \cite[Thm.~2.12]{DS23}. 
In particular, we can drastically simplify its original proof.

\subsection*{Notation}
Throughout, let $\Q$ denote the rational number field, 
$\N = \{1,2,\dots\}$ the strictly positive integers, and 
$\N_0 = \{0\}\cup\N$ the non-negative integers. 
\index[general]{Q@$\Q$}
\index[general]{N@$\N$}
\index[general]{N0@$\N_0$}
For every matrix $\mathbf{a}\in M^{m\times n}$ with entries from an arbitrary set $M$, we denote by $\size(\mathbf{a}):=(\rows(\mathbf{a}),\cols(\mathbf{a})):=(m,n)\in \N^2$ the number of its rows and columns, respectively. 
\index[general]{size@$\size$}
\index[general]{rows@$\rows$}
\index[general]{cols@$\cols$}
\index[general]{K@$\groundRing$}

Let $\groundRing$ be a commutative $\Q$-algebra and $(\mathfrak{M}_d,\star,\varepsilon)$ be the free commutative monoid  generated by $d$ letters $\{\w{1},\dots,\w{d}\}$. 
Here, $\star$ denotes the multiplication and $\varepsilon$ the neutral element with respect to the latter.  
From \cite[Sec.~2]{DS23} we recall the direct sum, 
 \[H:=\bigoplus_{\mathbf{a}\in\composition}\groundRing\mathbf{a}\] 
of matrix compositions, i.e., matrices with entries in $\mathfrak{M}_d$ which do not contain $\varepsilon$-rows or $\varepsilon$-columns. 
Additionally to those $\mathfrak{M}_d$-valued matrices, $\composition$ contains the empty composition denoted by $\ec$. 
Recall from \cite[Def.~2.1]{DS23} that any two matrix compositions $\mathbf{a}$ and $\mathbf{b}$ can be concatenated along the diagonal, $\mathbf{c}=\diag(\mathbf{a},\mathbf{b})$.  We denote by $\compositionConnected$ the subset of connected compositions, i.e., all $\mathbf{c}$ which have only trivial decompositions, that is $\mathbf{a}=\ec$ or $\mathbf{b}=\ec$. 
\index[general]{compCon@$\compositionConnected$}
\index[general]{ec@$\ec$}
\index[general]{H@H}
\index[general]{composition@$\composition$}
\index[general]{Md@$\mathfrak{M}_d$}
\index[general]{star@$\star$}
\index[general]{epsilon@$\varepsilon$}

\begin{remark}\label{rem:identificationdiagAndStringCon}
   We can identify $\composition$ with the Kleene star  of connected compositions $\compositionConnected^*$. For this compare \cite[Lem.~2.3]{DS23} which defines a one-to-one correspondence by sending the composition $\diag(\mathbf{u}_1,\dots,\mathbf{u}_{\ell})\in\composition$
   to the word $\mathbf{u}_1\dotsm\mathbf{u}_{\ell}\in\compositionConnected^*$.
   The empty composition $\ec$ is identified with  the empty word in $\compositionConnected^*$. 
   This identification requires to select the explicit coproduct $\Delta$ dual to diagonal concatenation, i.e. \cite[Def.~2.10]{DS23}. 
   Whenever we speak of such an identification, we implicitly make this choice and write $H=\groundRing\langle\compositionConnected\rangle$.
\end{remark}
\index[general]{KlComp@$\groundRing\langle\compositionConnected\rangle$}
\index[general]{compConStar@$\compositionConnected^*$}
\index[general]{diag@$\diag$}
\index[general]{Delta@$\Delta$}

 Let $\len\colon\compositionConnected^*\rightarrow\N_0$ denote the \DEF{length} of monomials. With the identification from \Cref{rem:identificationdiagAndStringCon} we have $\len(\ec)=0$ and $\len(\mathbf{a}):=\ell$ for every nonempty $\mathbf{a}=\diag(\mathbf{u}_1,\dots,\mathbf{u}_\ell)$ with $\mathbf{u}\in\compositionConnected^\ell$. 
 \index[general]{length@$\len:\compositionConnected^*\rightarrow\N_0$}

\section{The matrix composition Hopf algebra}

We begin with recalling the two-parameter shuffle from \cite{GLNO2022}, presented here in the language of 
\cite{DS23}.
Note that we define the operation in the larger $\groundRing$-module $H$, i.e., we use the extended alphabet $\compositionConnected$. 

\subsection{Two-parameter shuffles}
  For every $n\in\N$ let $\Sigma_n$ denote the \DEF{permutations} of $\{1,\dots,n\}$. 
  As in the classical setting, let $\sh(m,s)$ with $m,s\in\N$ denote the subset of permutations $q\in\Sigma_{m+s}$ 
such that $q(1)<\ldots<q(m)$ and $q(m+1)<\ldots<q(m+s)$.
  \index[general]{sh@$\sh$}
  \index[general]{Sigma@$\Sigma_n$}
  \index[general]{shuffle@$\shuffle$}
\begin{definition}\label{def:twodim_qshuffle}
The \DEF{two-parameter shuffle product} is the bilinear mapping
\[\shuffle\colon H\times H\rightarrow H\]  determined on nonempty $\groundRing$-basis elements 
$(\mathbf{a},\mathbf{b})\in\composition\times\composition$ via 
\[\mathbf{a}\shuffle\mathbf{b}
:=
\sum_{\substack{p\in\sh(\rows(\mathbf{a}),\rows(\mathbf{b}))\\q\in\sh(\cols(\mathbf{a}),\cols(\mathbf{b}))}}\;\;
{\left(\diag(\mathbf{a},\mathbf{b})_{p^{-1}(u),q^{-1}(v)}\right)}_{u,v}\]
and where $\mathbf{a}\shuffle\ec:=\ec\shuffle\mathbf{a}:=\mathbf{a}$ for all $\mathbf{a}\in\composition$. 
\end{definition}

\begin{example}
    \begin{align*}\begin{bmatrix}\w{1}\end{bmatrix}\shuffle\begin{bmatrix}\w{2}\end{bmatrix}
    &=
    \begin{bmatrix}\w{1}&\varepsilon\\
    \varepsilon&\w{2}\end{bmatrix}
    +
    \begin{bmatrix}\varepsilon&\w{1}\\
    \w{2}&\varepsilon\end{bmatrix}
    +
    \begin{bmatrix}\varepsilon&\w{2}\\
    \w{1}&\varepsilon\end{bmatrix}
    +
    \begin{bmatrix}\w{2}&\varepsilon\\
    \varepsilon&\w{1}\end{bmatrix}
    \\
\begin{bmatrix}\w{1}\end{bmatrix}\shuffle\begin{bmatrix}\w{2}&\w{3}\end{bmatrix}
&=
\begin{bmatrix}\w{1}&\varepsilon&\varepsilon\\
\varepsilon&\w{2}&\w{3}\end{bmatrix}
+
\begin{bmatrix}\varepsilon&\w{1}&\varepsilon\\
\w{2}&\varepsilon&\w{3}\end{bmatrix}
+
\begin{bmatrix}\varepsilon&\varepsilon&\w{1}\\
\w{2}&\w{3}&\varepsilon\end{bmatrix}\\
&\;\;\;\;\;\;+
\begin{bmatrix}\varepsilon&\w{2}&\w{3}\\
\w{1}&\varepsilon&\varepsilon\end{bmatrix}+
\begin{bmatrix}\w{2}&\varepsilon&\w{3}\\
\varepsilon&\w{1}&\varepsilon\end{bmatrix}
+
\begin{bmatrix}\w{2}&\w{3}&\varepsilon\\
\varepsilon&\varepsilon&\w{1}\end{bmatrix}
\end{align*}
\end{example}

\begin{theorem}
    $(H,\shuffle,\Delta)$ is a graded, connected, and commutative Hopf algebra.
\end{theorem}
\begin{proof}
   An elementary way of proving this is similar to \cite[Thm.~2.12]{DS23}. 
   Alternatively we can use the novel two-parameter version of Hoffman's isomorphism $\Phi$ developed in this article  (\Cref{thm_iso_sh_qsh}), or an analogous proof presented in \Cref{subsec:applic} via the one-parameter framework. 
\end{proof}

We can describe permutations in $\sh(m,s)\subseteq\Sigma_{m+s}$ with a one-hot matrix encoding, similar to \cite[Rem.~5.8]{DS23} for quasi-shuffles.  
With this we can understand the two-parameter shuffle product via suitable matrix actions on the row and column space\footnote{The matrix action on the row and column space in $\composition$ is defined in \cite[Sec.~5.2]{DS23}.}. 

\begin{remark}\label{rem:SH_121}
The following subset of permutation matrices 
\[\SH(m,s):=\left\{\begin{bmatrix}\e{\iota_1}&\cdots&\e{\iota_m}&\e{\kappa_1}&\cdots&\e{\kappa_s}\end{bmatrix}\in \Sigma_{m+s}\mid
\iota_p < \iota_{p+1}, 
\kappa_q<\kappa_{q+1}\right\}\]
is in one-to-one correspondence to $\sh(m,s)$. 
\end{remark}
\index[general]{SH@$\SH$}

\begin{example}
For  $p\in\sh(2,1)$ with $p(1)=1,p(2)=3$ and  $p(3)=2$, 
\[{\left(\begin{bmatrix}\w{1}&\varepsilon\\\w{2}&\varepsilon\\\varepsilon&\w{3}\end{bmatrix}_{p^{-1}(u),v}\right)}_{u,v}=
\begin{bmatrix}\w{1}&\varepsilon\\\varepsilon&\w{3}\\\w{2}&\varepsilon\end{bmatrix}_{u,v}=
{\left(
\begin{bmatrix}1&0&0\\0&0&1\\0&1&0\end{bmatrix}
\cdot\begin{bmatrix}\w{1}&\varepsilon\\\w{2}&\varepsilon\\\varepsilon&\w{3}\end{bmatrix}\right)}_{u,v.}
\]
\end{example}

\begin{lemma}\label{lem:shuffle_via_matrix_not}For all nonempty $\mathbf{a},\mathbf{b}\in\composition$, we have 
\[\mathbf{a}\shuffle\mathbf{b}=\sum_{\substack{\mathbf{P}\in\SH(\rows(\mathbf{a}),\rows(\mathbf{b}))\\\mathbf{Q}\in\SH(\cols(\mathbf{a}),\cols(\mathbf{b}))}}\;\;\mathbf{P}\diag(\mathbf{a},\mathbf{b})\mathbf{Q}^\top.\]
\end{lemma}

\subsection{Block shuffles}\label{sec:blockShuffles}

Let $(H,\blockShuffle,\Delta)$ denote the classical Hopf algebra $(H,\blockShuffle,\Delta)$ over the extended alphabet $\compositionConnected$ which shuffles the blocks (connected components with respect to $\Delta$) of matrix compositions.    
A formal way of defining this block shuffle is provided in \Cref{def:rec_def_blockShuffle}. 
\begin{example}
        \begin{align*}
\begin{bmatrix}\w{1}\end{bmatrix}\blockShuffle\begin{bmatrix}\w{2}\end{bmatrix}
    &=
\begin{bmatrix}\w{1}&\varepsilon\\
\varepsilon&\w{2}\end{bmatrix}
+
\begin{bmatrix}\w{2}&\varepsilon\\
\varepsilon&\w{1}\end{bmatrix}\\
\begin{bmatrix}\w{1}\end{bmatrix}\blockShuffle\begin{bmatrix}\w{2}&\w{3}\end{bmatrix}
&=\begin{bmatrix}\w{1}&\varepsilon&\varepsilon\\
\varepsilon&\w{2}&\w{3}\end{bmatrix}
+
\begin{bmatrix}\w{2}&\w{3}&\varepsilon\\
\varepsilon&\varepsilon&\w{1}\end{bmatrix}\\
\begin{bmatrix}\w{1}\end{bmatrix}\blockShuffle\begin{bmatrix}\w{2}&\varepsilon&\varepsilon\\
\varepsilon&\w{3}&\w{4}\end{bmatrix}
&=\begin{bmatrix}\w{1}&\varepsilon&\varepsilon&\varepsilon\\
\varepsilon&\w{2}&\varepsilon&\varepsilon\\
\varepsilon&\varepsilon&\w{3}&\w{4}\end{bmatrix}
+
\begin{bmatrix}\w{2}&\varepsilon&\varepsilon&\varepsilon\\
\varepsilon&\w{1}&\varepsilon&\varepsilon\\
\varepsilon&\varepsilon&\w{3}&\w{4}\end{bmatrix}
+\begin{bmatrix}\w{2}&\varepsilon&\varepsilon\\
\varepsilon&\w{3}&\w{4}&\varepsilon\\
\varepsilon&\varepsilon&\varepsilon&\w{1}\end{bmatrix}
\end{align*}
\end{example}

\begin{definition}
    \label{def:rec_def_blockShuffle}
    The \DEF{block shuffle product} is the bilinear mapping
\[\blockShuffle:H\times H\rightarrow H, \] 
which satisfies the recursion 
\[\diag(\mathbf{a},\mathbf{u})\blockShuffle\diag(\mathbf{b},\mathbf{v})
=\diag(\mathbf{a},\mathbf{u}\blockShuffle\diag(\mathbf{b},\mathbf{v}))+\diag(\mathbf{b},\diag(\mathbf{a},\mathbf{u})\blockShuffle\mathbf{v})\]
for $\mathbf{a},\mathbf{b}\in\compositionConnected$ and $\mathbf{u},\mathbf{v}\in\composition$,
and such that
the empty matrix composition $\ec$ is the neutral element with respect to this product. 
\end{definition}
\index[general]{blockShuffle@$\blockShuffle$}
With this we obtain the (classical, one-parameter) shuffle algebra over the extended alphabet $\compositionConnected$. 
\begin{theorem}\label{thm:1paramHopfAlg}
    $(H,\blockShuffle,\Delta)$ is a graded, connected, and commutative Hopf algebra.
\end{theorem}

As in \Cref{lem:shuffle_via_matrix_not}, we encode block shuffles via matrix notation. 
For this, we index elements due to \Cref{rem:SH_121},  \[\SH(a,b)=
    \left\{\begin{bmatrix}
            \e{\sigma^{-1}(1)}^\top\\
            \vdots\\
            \e{\sigma^{-1}(a+b)}^\top
            \end{bmatrix}\mid \sigma\in\sh(a,b)\right\}\]
   via its rows. 
   To encode block shuffles
     let $\otimes$ be the Kronecker product of matrices and 
            \[\bSH_\sigma(u,v):=
            \begin{bmatrix}
        \e{\sigma^{-1}(1)}^\top\otimes\mathrm{I}_{u_1}\\
            \vdots\\
            \e{\sigma^{-1}(a)}^\top\otimes\mathrm{I}_{u_a}\\
            \e{\sigma^{-1}(a+1)}^\top\otimes\mathrm{I}_{v_1}\\
            \vdots\\
            \e{\sigma^{-1}(a+b)}^\top\otimes\mathrm{I}_{v_b}
\end{bmatrix}\in\SH\left(\sum_{1\leq\alpha\leq a}u_\alpha,\sum_{1\leq\beta\leq b}v_\beta\right)\]
 be a \DEF{block shuffle action matrix} for every $\sigma\in\sh(a,b)$ and block sizes $u=(u_1,\dots, u_a)$ and $v=(v_1,\dots, v_b)$. 
 \index[general]{blockShuffleMatrix@$\bSH$}

\begin{lemma}\label{lem:blockShuffleViaMatrices}
For $\mathbf{a}=\diag(\mathbf{u}_1,\dots,\mathbf{u}_a)$ and $\mathbf{b}=\diag(\mathbf{v}_1,\dots,\mathbf{v}_b)$ with decompositions into connected compositions $\mathbf{u}$ and $\mathbf{v}$, we have 
\[\mathbf{a}\blockShuffle\mathbf{b}=\sum_{\sigma\in\sh(a,b)}\;\;\bSH_\sigma(\rows(\mathbf{u}),\rows(\mathbf{v}))\,\diag(\mathbf{a},\mathbf{b})\,\bSH_\sigma(\cols(\mathbf{u}),\cols(\mathbf{v}))^\top\]
where $\rows$ and $\cols$ is applied on the components of $\mathbf{u}$ and $\mathbf{v}$, respectively. 
\end{lemma}

We can write the two-parameter shuffle product (\Cref{def:twodim_qshuffle}) as a block shuffle (\Cref{def:rec_def_blockShuffle}) plus some lower order terms with respect to the length in $\composition$.

\begin{lemma}\label{lem:block_shuffle_and_2param_shuffle}
For all $\mathbf{a},\mathbf{b}\in\composition$,     \[\mathbf{a}\shuffle\mathbf{b}=\mathbf{a}\blockShuffle\mathbf{b}+\sum_{1\leq i\leq \ell}\mathbf{c}_i\]
where $\mathbf{c}_i\in\composition$ with $\len(\mathbf{c}_i)<\len(\mathbf{a})+\len(\mathbf{b})$ for all $1\leq i
    \leq 
    \ell$.
\end{lemma}
\begin{proof}
    Follows with \Cref{lem:shuffle_via_matrix_not,lem:blockShuffleViaMatrices}. 
\end{proof}

\subsection{Two-parameter quasi-shuffles}
 As in the classical setting, let $\qSh(m,s;j)$ with $m,s,j\in\N$ denote the set of surjections
\[q:\{1,\ldots,\,m+s\}\twoheadrightarrow\{1,\ldots,\,j\}\] 
such that $q(1)<\ldots<q(m)$ and $q(m+1)<\ldots<q(m+s)$.
  \index[general]{qSh@$\qSh$}
  \index[general]{quasiShuffleProduct@$\qShuffle$}
From \cite[Def.~2.8]{DS23} we recall the \DEF{two-parameter quasi-shuffle product} 
\begin{equation}\label{eq:qShuffle}
    \mathbf{a}\qShuffle\mathbf{b}
=\sum_{j,k\in\N}\;\; 
\sum_{\substack{p\in\qSh(m,s;j)\\q\in\qSh(n,t;k)}}\;\;
\begin{bmatrix}
\mathbf{c}^{p,q}_{1,1}&\ldots&\mathbf{c}^{p,q}_{1,k}\\
\vdots&&\vdots\\
\mathbf{c}^{p,q}_{j,1}&\ldots&\mathbf{c}^{p,q}_{j,k}\\
\end{bmatrix}
\end{equation}
where
\begin{equation*}
\mathbf{c}^{p,q}_{x,y}=\underset{\substack{u\in p^{-1}(x)\\v\in q^{-1}(y)}}\bigstar{\diag(\mathbf{a},\mathbf{b})}_{u,v}\in\monoidComp_d.
\end{equation*}
As discussed in \cite[Rem.~5.8 and Lem.~5.11]{DS23}, we know that $\qSh(m,s;j)$ is in one-to-one correspondence to 
\index[general]{QSH@$\QSH$}
\[\QSH(m,s;j):=\left\{\begin{bmatrix}\e{\iota_1}&\cdots&\e{\iota_m}&\e{\kappa_1}&\cdots&\e{\kappa_s}\end{bmatrix}\in\N_0^{j\times(m+s)}\;
\begin{array}{|l}
  \iota_{1}<\cdots<\iota_{m} \\
  \kappa_{1}<\cdots<\kappa_s \\
  \text{right invertible}
\end{array}\right\}\]
and that \Cref{eq:qShuffle} translates to 
\[\mathbf{a}\qShuffle\mathbf{b}
=\sum_{j,k\in\N}\;\; 
\sum_{\substack{\mathbf{P}\in\QSH(m,s;j)\\\mathbf{Q}\in\QSH(n,t;k)}}
\mathbf{P}\diag(\mathbf{a},\mathbf{b}) \mathbf{Q}^\top.
\]
From \cite[Thm.~2.12]{DS23} we know that $(H,\qShuffle,\Delta)$ is a graded, connected, and commutative Hopf algebra. 

Similar as in \Cref{lem:block_shuffle_and_2param_shuffle} for two-parameter shuffles, we can write \Cref{eq:qShuffle} as a block shuffle plus some lower order terms with respect to the length in $\composition$. 
\begin{lemma}\label{lem:block_shuffle_and_2param_qShuffle}
For all nonempty $\mathbf{a},\mathbf{b}\in\composition$,     \[\mathbf{a}\qShuffle\mathbf{b}=\mathbf{a}\blockShuffle\mathbf{b}+\sum_{1\leq i\leq \ell}\mathbf{c}_i\]
where $\mathbf{c}_i\in\composition$ with $\len(\mathbf{c}_i)<\len(\mathbf{a})+\len(\mathbf{b})$ for all $1\leq i
    \leq 
    \ell$.
\end{lemma}

\section{Free Lyndon generators}
\label{sec:Lyndongens}

\subsection{Total order on connected compositions}

 We equip $\compositionConnected$ with a total order, e.g. part \ref{def:total_order2}.~of \Cref{def:total_order}. 
Note that this explicit order is arbitrary, and can be exchanged by  any other total order. 
For convenience, this explicit choice is fixed throughout this section, and illustrated in \Cref{ex:total_order}.

\begin{definition}\label{def:total_order}~
\begin{enumerate}
    \item\label{def:total_order1}
    Let $(\mathfrak{M}_d,>)$ be ordered degree-lexicoraphically with $\w{1}<\dots<\w{d}$. 
For arbitrary $n\in\N$, extend this lexicographically to $(\mathfrak{M}_d^n,>)$  via
\begin{equation*}
\mathbf{p}>\mathbf{q}:=
\begin{cases}\mathbf{p}_1>\mathbf{q}_1\text{ or }\\
\mathbf{p}_i=\mathbf{q}_i\text{ for all }i\leq j-1\text{ and }{\mathbf{p}}_j>{\mathbf{q}}_{j},
\end{cases}
\end{equation*}
where $\mathbf{p},\mathbf{q}\in \mathfrak{M}_d^n$. 
\item\label{def:total_order2}
For every $n\in\N$, let $\left(\mathfrak{M}_d^n,>\right)$ be ordered according to part \ref{def:total_order1}.  
Using this, we equip $(\compositionConnected,>)$ with a strict total order 
 via     \[\mathbf{a}>\mathbf{b}:=\begin{cases}
        \rows(\mathbf{a})>\rows(\mathbf{b})\text{ or}
        \\
        (
        \rows(\mathbf{a})=\rows(\mathbf{b}) \text{ and }\cols(\mathbf{a})>\cols(\mathbf{b}))\text{ or}
        \\
        \size(\mathbf{a})=\size(\mathbf{b}) \text{ and }\vectorize(\mathbf{a})>\vectorize(\mathbf{b}),
    \end{cases}\]
    where   $\mathbf{a},\mathbf{b}\in\compositionConnected$. Furthermore, let $\mathbf{a}>\ec$ for every $\mathbf{a}$. 
    
    \end{enumerate}
\end{definition}

\begin{example}\label{ex:total_order}
    For $>$ according to part \ref{def:total_order2}.~of \Cref{def:total_order}, we
    have for instance 
    \[\begin{bmatrix}
    \w{1}
\end{bmatrix}
<
\begin{bmatrix}
    \w{2}
\end{bmatrix}
<
\begin{bmatrix}
    \w{1}\star\w{2}\star\w{3}
\end{bmatrix}
<
\begin{bmatrix}
    \w{1}&\w{2}
\end{bmatrix}
<
\begin{bmatrix}
    \w{2}&\w{1}
\end{bmatrix}
<
\begin{bmatrix}
    \w{1}\\\w{2}
\end{bmatrix}
<
\begin{bmatrix}
    \w{2}\\\w{1}
\end{bmatrix}
<
\begin{bmatrix}
    \w{1}&\varepsilon\\\varepsilon&\w{2}
\end{bmatrix}.
\]
\end{example}

Having a total order on $\compositionConnected$ we can extend it to  $\compositionConnected^*$. 
Note that we do not yet identify $\compositionConnected^*$ with $\composition$. 

\index[general]{lex@$\leq_{\lex}$}
\index[general]{grlex@$<_{\grlex}$}
\begin{definition}\label{def:order}
Let $(\compositionConnected,>)$ be totally ordered. 
\begin{enumerate}
\item We recall the \DEF{lexicographic order} $\leq_{\lex}$ on $\compositionConnected^*$ via 
\[u\leq_{\lex}v:=
\begin{cases}
   \text{either } \exists i\leq\min\{\len(u),\len(v)\}:u_i<v_i\land u_j=v_j\;\forall j\leq i-1\\
   \text{or }\exists a\in \compositionConnected^*:v=ua. 
\end{cases}\]
\item\label{def:order2} Let the \DEF{graded lexicographic order} on $\compositionConnected^*$ be denoted by $<_{\grlex}$, that is  \[u<_{\grlex}v:=\begin{cases}\len(u)<\len(v)\text{ or}\\
    \len(u)=\len(v)\text{ and  }u<_{\lex}v.
    \end{cases}\]
    \end{enumerate}
\end{definition}

The order $<_{\grlex}$ in part \ref{def:order2}.~of \Cref{def:order}  is \DEF{$\len$-compatible} by construction, i.e., elements are first ordered by its length, and afterwards lexicographically to break ties. 
Furthermore it is a \DEF{monomial order}, i.e., 
\[a<_{\grlex}b\implies uav<_{\grlex}ubv\]
for all $a,b,u,v\in\compositionConnected^*$. 

\begin{example}\label{ex:order}
We continue with \Cref{ex:total_order}, i.e., we have a total order on $\compositionConnected$ according to \Cref{def:total_order}.
    Let $<_{\grlex}$ be due to \Cref{def:order}, e.g., 
\[\begin{bmatrix}
    \w{1}
\end{bmatrix}
<_{\grlex}
\begin{bmatrix}
    \w{1}\star\w{2}\star\w{3}
\end{bmatrix}
<_{\grlex}
\begin{bmatrix}
    \w{1}&\varepsilon\\\w{3}&\w{2}
\end{bmatrix}<_{\grlex}
\begin{bmatrix}
    \w{1}
\end{bmatrix}
\begin{bmatrix}
    \w{2}
\end{bmatrix}
<_{\grlex}
\begin{bmatrix}
    \w{2}
\end{bmatrix}
\begin{bmatrix}
    \w{1}
\end{bmatrix}
<_{\grlex}
\begin{bmatrix}
    \w{1}\\\w{2}
\end{bmatrix}
\begin{bmatrix}
    \w{3}
\end{bmatrix}
\begin{bmatrix}
    \w{4}
\end{bmatrix}.
\]
\end{example}

While the order of $\compositionConnected$ from \Cref{def:total_order} was arbitrary, 
the way \Cref{def:order} extends it to $\compositionConnected^*$ is crucial for the following section. 

\subsection{Lyndon words}
A word 
$w\in\compositionConnected^*$ is a \DEF{Lyndon word}, if and only if it is nonempty and lexicographically strictly smaller than any of its proper suffixes, that is 
   $ w
<_{\lex}
v$
for all nonempty words 
$v\in\compositionConnected^*$
 such that 
$w
=
u
v$
and 
$u\in\compositionConnected^*$ is nonempty.
Let 
 $\mathfrak{L}=\mathfrak{L}(\compositionConnected,<_{\lex})$ be the set of Lyndon words over the alphabet $\compositionConnected$.
 We define a family  ${(\mathcal{B}_w)}_{w\in\mathfrak{L}}$ of compositions  
via
\begin{equation}\label{eq:def_B_lyndon}
       \mathcal{B}_w:=\diag(\mathbf{w}_1,\dots,\mathbf{w}_{\len(w)})
   \end{equation}
for every 
$w=\mathbf{w}_1\dots\mathbf{w}_{\len(w)}\in\mathfrak{L}$. 
\index[general]{L@$\mathfrak{L}$}
\index[general]{B@$\mathcal{B}$}

 \begin{example}\label{ex:lyndon}
In the setting of \Cref{ex:order}, we have for instance 
\[\begin{bmatrix}
    \w{1}
\end{bmatrix},
\begin{bmatrix}
    \w{2}
\end{bmatrix},
\begin{bmatrix}
    \w{1}
\end{bmatrix}\begin{bmatrix}
    \w{2}
\end{bmatrix},
\begin{bmatrix}
    \w{3}
\end{bmatrix}
\begin{bmatrix}
    \w{1}\\\w{2}
\end{bmatrix},
\begin{bmatrix}
    \w{1}&\w{2}
\end{bmatrix}
\begin{bmatrix}
    \w{1}\\\w{2}
\end{bmatrix}
\in\mathfrak{L},
\]
and as non-examples, 
\[
\begin{bmatrix}
    \w{2}
\end{bmatrix}\begin{bmatrix}
    \w{1}
\end{bmatrix},
\begin{bmatrix}
    \w{1}\\\w{2}
\end{bmatrix}
\begin{bmatrix}
    \w{3}
\end{bmatrix},
\begin{bmatrix}
    \w{1}\\\w{2}
\end{bmatrix}
\begin{bmatrix}
    \w{1}&\w{2}
\end{bmatrix}, 
\begin{bmatrix}
    \w{1}\\\w{2}
\end{bmatrix}
\begin{bmatrix}
    \w{3}
\end{bmatrix}
\begin{bmatrix}
    \w{4}
\end{bmatrix}
\not\in\mathfrak{L}. 
\]
 \end{example}

In \Cref{thm:iso} we show that ${(\mathcal{B}_w)}_{w\in\mathfrak{L}}$ is a family of free generators  for the block shuffle algebra, as well as for two-parameter shuffle algebra, and for the two-parameter quasi-shuffle algebra.
\Cref{def:B_words} treats these three products from above in parallel, and requires the so-called CFL factorization of words.

 For nonempty $u\in\compositionConnected^*$
 let $(a_1,\dots,a_p)$ be its \DEF{Chen-Fox-Lyndon (CFL) factorization}, i.e., we have   $u=a_1\dots a_p$ and $a_1\geq_{\lex}\dots\geq_{\lex} a_p$ for $a_i\in\mathfrak{L}$ with $1\leq i\leq p$. This factorization always exists and it is unique, see e.g. \cite[Thm.~6.1.27]{grinberg2020hopf}.

\begin{definition}\label{def:B_words}
 Let $(H,m,+)$ be a commutative $\groundRing$-algebra. For every nonempty word $u\in\compositionConnected^*$ we 
    define \[\mathcal{B}_u^m:=m(\mathcal{B}_{a_1},\dots ,\mathcal{B}_{a_p})\in H\] 
    where $(a_1,\dots,a_p)$ is the CFL factorization of $u$, and $\mathcal{B}_{a_i}$ is according to \Cref{eq:def_B_lyndon}. If $u$ is empty, we define $\mathcal{B}^m_{u}:=\ec$. 
   \end{definition}
   \index[general]{Bm@${\mathcal{B}}^m$}
 
Note that for every Lyndon word $w\in\mathfrak{L}$ the identification in \Cref{rem:identificationdiagAndStringCon} and $\mathcal{B}_w$ from \cref{eq:def_B_lyndon} provide the same element in $H$. 
Otherwise, 
   if $u=\mathbf{u}_1\dots\mathbf{u}_{\len(u)}\not\in\mathfrak{L}$ is nonempty with $\mathbf{u}_i\in\compositionConnected$ for $1\leq i\leq \len(u)$, then \Cref{def:B_words} yields $\mathcal{B}_u^m$, which is distinct from $\diag(\mathbf{u}_1,\dots,\mathbf{u}_{\len(u)})$. 

\begin{example}
We continue with the framework of \Cref{ex:lyndon}. 
\begin{enumerate}
    \item For Lyndon word $w=\begin{bmatrix}
    \w{3}
\end{bmatrix}
    \begin{bmatrix}
    \w{1}\\\w{2}
\end{bmatrix}
\in\mathfrak{L}$
we have 
\[\mathcal{B}^m_w=\mathcal{B}_w=\begin{bmatrix}
    \w{3}&\varepsilon\\\varepsilon&\w{1}\\\varepsilon&\w{2}
\end{bmatrix}\]
regardless of the product $m$ in $H$.
\item 
For $u=
    \begin{bmatrix}
    \w{1}\\\w{2}
\end{bmatrix}
\begin{bmatrix}
    \w{3}
\end{bmatrix}
\in\compositionConnected^*$ we have $a=(a_1,a_2)=(\begin{bmatrix}
    \w{1}\\\w{2}
\end{bmatrix},\begin{bmatrix}
    \w{3}
\end{bmatrix})\in\mathfrak{L}^2$. 

In the algebras $(H,\blockShuffle,+)$ and $(H,\shuffle,+)$ we have
\[\mathcal{B}_u^{\;\blockShuffle\;}=\mathcal{B}_{a_1}\blockShuffle\mathcal{B}_{a_2}=\begin{bmatrix}
    \w{1}\\\w{2}
\end{bmatrix}\blockShuffle\begin{bmatrix}
    \w{3}
\end{bmatrix}
\quad\text{ and }\quad\mathcal{B}_u^{\;\shuffle\;}=\mathcal{B}_{a_1}\shuffle\mathcal{B}_{a_2}=\begin{bmatrix}
    \w{1}\\\w{2}
\end{bmatrix}\shuffle\begin{bmatrix}
    \w{3}
\end{bmatrix}. 
\]
%
\end{enumerate}
\end{example}

  From now on, we identify $\compositionConnected^*$ with $\composition$ as discussed in  \Cref{rem:identificationdiagAndStringCon}. 
  Then, $<_{\grlex}$ extends to the graded lexicoraphic order in $H=\groundRing\langle\compositionConnected\rangle$.
Let $\lm:H\setminus\{0\}\rightarrow\composition$ denote the \DEF{leading monomial} with respect to $<_{\grlex}$.
\index[general]{lm@$\lm$}

\begin{corollary}\label{cor:lm_sh_qSh_blockSh}
    For every $a,b\in H\setminus\{0\}$, we have \[\lm(a\blockShuffle b)=\lm(a\shuffle b)=\lm(a\qShuffle b).\]
\end{corollary}
\begin{proof}
    This is an immediate consequence of \Cref{lem:block_shuffle_and_2param_shuffle,lem:block_shuffle_and_2param_qShuffle}. 
\end{proof}

%


For nonempty $\mathbf{a}\in\composition$ let   
   $\mathcal{B}_{\mathbf{a}}^m:=\mathcal{B}^m_{\mathbf{u}_1\dots\mathbf{u}_{\len(\mathbf{a})}}$ be according to   \Cref{def:B_words}, and where $\mathbf{u}$ is the decomposition  of $\mathbf{a}$ into connected compositions.   Furthermore, let $\mathcal{B}_{\ec}^m:=\ec$.

\begin{lemma}\label{lem:lmBa_is_a}
In the three $\groundRing$-algebras  $(H,m,+)$ with $m\in\{\blockShuffle,\shuffle,\qShuffle\}$ we have  $$\lm(\mathcal{B}_{\mathbf{a}}^m)=\mathbf{a}$$
    for every $\mathbf{a}\in\composition$.
\end{lemma}
\begin{proof}
Clearly $\lm(\ec)=\ec$, hence we can assume that $\mathbf{a}$ is nonempty. 
    We use induction over the length of $\mathbf{a}$, i.e., we show 
    \[\lm(\mathcal{B}_{{\mathbf{u}}_1\dots\mathbf{u}_n}^m)=\diag(\mathbf{u}_1,\dots,\mathbf{u}_n)\]
    for every $n\in\N$ and $\mathbf{u}_i\in\compositionConnected$ with $1\leq i\leq n$. 
    The base case $n=1$ is clear with  $\compositionConnected\subseteq\mathfrak{L}$. 
    For the induction step, let $(a_1,\dots,a_p)$ be the CFL factorization of a word $\mathbf{u}_1\dots\mathbf{u}_n$ with length $n$. The case $p=1$ is clear again with $\compositionConnected\subseteq\mathfrak{L}$. Assuming $p>1$, the CFL decompositions of  $u:=a_1$ and $v:=a_2\dots a_p$ are $a_1$ and $(a_2,\dots,a_p)$ respectively. 
    Clearly $u\in\mathfrak{L}$
    hence $\mathcal{B}^m_u=\diag(\mathbf{u}_1,\dots,\mathbf{u}_\ell)$ for a unique $\ell\in\N$. Furthermore $\len(v)<n$, thus inductively \begin{equation}
        \lm(\mathcal{B}_{v}^m)=\lm(\mathcal{B}^m_{{\mathbf{u}}_{\ell+1}\dots\mathbf{u}_n})=\diag(\mathbf{u}_{\ell+1},\dots,\mathbf{u}_{n}).\end{equation}
    Now we have to distinguish between the underlying product $m$.
    \begin{enumerate}
        \item\label{proof:lmBa_case1_blockShuffle} If  $m=\blockShuffle$ then \begin{align*}\mathcal{B}_{{\mathbf{u}}_1\dots\mathbf{u}_n}^{\;\blockShuffle}
        &=\mathcal{B}_{a_1}^{\;\blockShuffle}\blockShuffle\dots\blockShuffle\mathcal{B}_{a_p}^{\;\blockShuffle}
        =\mathcal{B}_{u}^{\;\blockShuffle}\blockShuffle\mathcal{B}_{v}^{\;\blockShuffle}\\
        &=\diag(\mathbf{u}_1,\dots,\mathbf{u}_\ell)\blockShuffle\diag(\mathbf{u}_{\ell+1},\dots,\mathbf{u}_{n})\end{align*}
        thus with 
        a basic fact \cite[Thm.~6.2.2]{grinberg2020hopf} about Lyndon words, \begin{align*}\lm(\mathcal{B}_{{\mathbf{u}}_1\dots\mathbf{u}_n}^{\;\blockShuffle})
        &=\lm(\diag(\mathbf{u}_1,\dots,\mathbf{u}_\ell)\blockShuffle\diag(\mathbf{u}_{\ell+1},\dots,\mathbf{u}_n))\\
        &=\diag(\mathbf{u}_1,\dots,\mathbf{u}_n).
        \end{align*}
        \item In the remaining cases $m\in\{\shuffle,\qShuffle\}$ we use  $p-1$ times \Cref{cor:lm_sh_qSh_blockSh}, and also part~\ref{proof:lmBa_case1_blockShuffle}.,
        that is  \begin{align*}
        \lm(\mathcal{B}_{{\mathbf{u}}_1\dots\mathbf{u}_n}^{m})=
        \lm(\mathcal{B}_{{\mathbf{u}}_1\dots\mathbf{u}_n}^{\;\blockShuffle})
        =
        \diag(\mathbf{u}_1,\dots,\mathbf{u}_n). 
        \end{align*}
    \end{enumerate}
\end{proof}

%

\begin{corollary}\label{lem:bases}
Via the $\groundRing$-algebras $(H,m,+)$ with $m\in\{\blockShuffle,\shuffle,\qShuffle\}$ we obtain three distinct $\groundRing$-bases ${(\mathcal{B}^{m}_{\mathbf{a}})}_{\mathbf{a}\in\composition}$ for the underlying $\groundRing$-module $(H,+)$. 
\end{corollary}
\begin{proof}
We know that 
    $\composition$ is a $\groundRing$-basis of $H$. 
\Cref{lem:lmBa_is_a} implies that ${(\mathcal{B}^m_{\mathbf{a}})}_{\mathbf{a}\in\composition}$ expands invertibly triangularly with respect to 
 $\composition$ for every $m$so it is a $\groundRing$-basis, too.  
\end{proof}

We recall another fact about Lyndon words.

\begin{lemma}\label{lem:grinberg}  
    The family ${(\mathcal{B}_w)}_{w\in\mathfrak{L}}$  freely generates $(H,m,+)$ if and only if ${(\mathcal{B}^m_{\mathbf{a}})}_{\mathbf{a}\in\composition}$ is a $\groundRing$-basis of $(H,+)$. 
\end{lemma}
\begin{proof}
Consider for instance  \cite[Lem.~6.3.7]{grinberg2020hopf}. 
\end{proof}

\begin{theorem}\label{thm:iso}
The family ${(\mathcal{B}_w)}_{w\in\mathfrak{L}}$ 
freely generates the three  $\groundRing$-algebras $(H,\blockShuffle,+)$,   $(H,\shuffle,+)$ and   $(H,\qShuffle,+)$. 
\end{theorem}
\begin{proof}
Follows with \Cref{lem:bases} and \Cref{lem:grinberg}.
\end{proof}

Tables \ref{tab:gens_in_block_shuffles} and \ref{tab:gens_in_shuffles}
provide examples of non-trivial expressions in terms of their generators. 
One way of producing these is by using \Cref{lem:bases} and $<_{\grlex}$ to subtract leading terms by  ${\mathcal{B}}_u^m$ when considered as monomials in ${\mathcal{B}}_w$ for $w\in\mathfrak{L}$.

\begin{table}[ht]
\begin{small}
    \centering
    \begin{tabular}{r|l}
    $\composition$ & $(H,\blockShuffle,+)$ in terms of its free generators  $(\mathcal{B}_{w})_{w\in\mathfrak{L}}$\\
    \hline\\
       $\begin{bmatrix}
           \w{2}&\varepsilon\\
           \varepsilon&\w{1}
       \end{bmatrix}$
       &  
       $\begin{bmatrix}
           \w{1}
       \end{bmatrix}\blockShuffle\begin{bmatrix}
           \w{2}
       \end{bmatrix}
       -\begin{bmatrix}
           \w{1}&\varepsilon\\
           \varepsilon&\w{2}
       \end{bmatrix}$
       \\\\
$\begin{bmatrix}
           \w{1}&\varepsilon\\
           \varepsilon&\w{1}
       \end{bmatrix}$
       &
       $\frac12\begin{bmatrix}
           \w{1}
       \end{bmatrix}\blockShuffle\begin{bmatrix}
           \w{1}
       \end{bmatrix}
       $
       \\\\
       $\begin{bmatrix}
           \w{1} & \w{2} & \varepsilon\\
          \varepsilon&\varepsilon& \w{3}
       \end{bmatrix}$  
       & 
       $\begin{bmatrix}
           \w{1} & \w{2}
       \end{bmatrix}\blockShuffle\begin{bmatrix}
           \w{3}
       \end{bmatrix}
       -\begin{bmatrix}
           \w{3}&\varepsilon&\varepsilon\\
           \varepsilon&\w{1}&\w{2}
       \end{bmatrix}$
       \\\\
       $\begin{bmatrix}
          \varepsilon  & \w{1} & \varepsilon\\
        \w{1}  & \varepsilon & \varepsilon\\
          \varepsilon&\varepsilon& \w{1}
       \end{bmatrix}$  
       & 
       $\begin{bmatrix}
           \varepsilon &\w{1}\\
            \w{1} & \varepsilon
       \end{bmatrix}\blockShuffle\begin{bmatrix}
           \w{1} 
       \end{bmatrix}
       -
       \begin{bmatrix}
           \w{1} & \varepsilon & \varepsilon\\ \varepsilon& \varepsilon& \w{1} \\
          \varepsilon& \w{1}&\varepsilon
       \end{bmatrix}$
       \\\\
       $\begin{bmatrix}
           \w{2} & \varepsilon & \varepsilon\\ \varepsilon& \w{3} & \varepsilon\\
          \varepsilon&\varepsilon& \w{1}
       \end{bmatrix}$  
       & 
       $\begin{bmatrix}
           \w{2} &\varepsilon\\\varepsilon& \w{3}
       \end{bmatrix}\blockShuffle\begin{bmatrix}
           \w{1} 
       \end{bmatrix}
       -
       \begin{bmatrix}
           \w{2} 
       \end{bmatrix}\blockShuffle\begin{bmatrix}
           \w{1} &\varepsilon\\\varepsilon& \w{3}
       \end{bmatrix}
       +
       \begin{bmatrix}
           \w{1} & \varepsilon & \varepsilon\\ \varepsilon& \w{3} & \varepsilon\\
          \varepsilon&\varepsilon& \w{2}
       \end{bmatrix}$
       \\\\
       $\begin{bmatrix}
           \w{3} & \varepsilon & \varepsilon\\ \varepsilon& \w{1} & \varepsilon\\
          \varepsilon&\varepsilon& \w{2}
       \end{bmatrix}$  
       & 
       $\begin{bmatrix}
           \w{3} 
       \end{bmatrix}\blockShuffle\begin{bmatrix}
           \w{1} &\varepsilon\\\varepsilon& \w{2}
       \end{bmatrix}
       -
       \begin{bmatrix}
           \w{1} & \varepsilon & \varepsilon\\ \varepsilon& \w{2} & \varepsilon\\
          \varepsilon&\varepsilon& \w{3}
       \end{bmatrix}
       -
       \begin{bmatrix}
           \w{1} & \varepsilon & \varepsilon\\ \varepsilon& \w{3} & \varepsilon\\
          \varepsilon&\varepsilon& \w{2}
       \end{bmatrix}$
       \\\\
       $\begin{bmatrix}
           \w{3} & \varepsilon & \varepsilon\\ \varepsilon& \w{2} & \varepsilon\\
          \varepsilon&\varepsilon& \w{1}
       \end{bmatrix}$  
       & 
       $
       \begin{bmatrix}
           \w{1} 
       \end{bmatrix}\blockShuffle\begin{bmatrix}
           \w{2} 
       \end{bmatrix}\blockShuffle\begin{bmatrix}
           \w{3} 
       \end{bmatrix}
       -
       \begin{bmatrix}
           \w{2} &\varepsilon\\\varepsilon& \w{3}
       \end{bmatrix}\blockShuffle\begin{bmatrix}
           \w{1} 
       \end{bmatrix}
       -
       \begin{bmatrix}
           \w{3} 
       \end{bmatrix}\blockShuffle\begin{bmatrix}
           \w{1} &\varepsilon\\\varepsilon& \w{2}
       \end{bmatrix}
       +
       \begin{bmatrix}
           \w{1} & \varepsilon & \varepsilon\\ \varepsilon& \w{2} & \varepsilon\\
          \varepsilon&\varepsilon& \w{3}
       \end{bmatrix}$
       \\\\
       $\begin{bmatrix}
           \w{1} & \varepsilon & \varepsilon\\ \varepsilon& \w{2} & \varepsilon\\
          \varepsilon&\varepsilon& \w{1}
       \end{bmatrix}$  
       & 
       $\begin{bmatrix}
           \w{1} &\varepsilon\\\varepsilon& \w{2}
       \end{bmatrix}\blockShuffle\begin{bmatrix}
           \w{1} 
       \end{bmatrix}
       -2
       \begin{bmatrix}
           \w{1} & \varepsilon & \varepsilon\\ \varepsilon& \w{1} & \varepsilon\\
          \varepsilon&\varepsilon& \w{2}
       \end{bmatrix}$
    \end{tabular}
    \caption{Block shuffles}
    \label{tab:gens_in_block_shuffles}
    \end{small}
\end{table}

\begin{table}[ht]
    \centering
    \begin{small}
    \begin{tabular}{r|l}
    $\composition$ & $(H,\shuffle,+)$ in terms of its free generators  $(\mathcal{B}_{w})_{w\in\mathfrak{L}}$\\
    \hline\\
       $\begin{bmatrix}
           \w{2}&\varepsilon\\
           \varepsilon&\w{1}
       \end{bmatrix}$
       &  
       $\begin{bmatrix}
           \w{1}
       \end{bmatrix}\shuffle\begin{bmatrix}
           \w{2}
       \end{bmatrix}
       -\begin{bmatrix}
           \w{1}&\varepsilon\\
           \varepsilon&\w{2}
       \end{bmatrix}
       -\begin{bmatrix}
           \varepsilon&\w{1}\\
           \w{2}&\varepsilon
       \end{bmatrix}
       -\begin{bmatrix}
           \varepsilon&\w{2}\\
           \w{1}&\varepsilon
       \end{bmatrix}$
       \\\\
$\begin{bmatrix}
           \w{1}&\varepsilon\\
           \varepsilon&\w{1}
       \end{bmatrix}$
       &
       $\frac12\begin{bmatrix}
           \w{1}
       \end{bmatrix}\shuffle\begin{bmatrix}
           \w{1}
       \end{bmatrix} -
       \begin{bmatrix}
           \varepsilon&\w{1}\\
           \w{1}&\varepsilon
       \end{bmatrix}
       $
       \\\\
       $\begin{bmatrix}
           \w{1} & \w{2} & \varepsilon\\
          \varepsilon&\varepsilon& \w{3}
       \end{bmatrix}$  
       & 
       $\begin{bmatrix}
           \w{1} & \w{2}
       \end{bmatrix}\shuffle\begin{bmatrix}
           \w{3}
       \end{bmatrix}
       -
       \begin{bmatrix}
           \w{3}&\varepsilon&\varepsilon\\
           \varepsilon&\w{1}&\w{2}
       \end{bmatrix}
       - 
\begin{bmatrix}\w{1}&\varepsilon&\w{2}\\
    \varepsilon&\w{3}&\varepsilon\end{bmatrix}
-
\begin{bmatrix}\varepsilon&\varepsilon&\w{3}\\
    \w{1}&\w{2}&\varepsilon\end{bmatrix}
    -
    \begin{bmatrix}\varepsilon&\w{3}&\varepsilon\\
    \w{1}&\varepsilon&\w{2}\end{bmatrix}
    -
    \begin{bmatrix}
    \varepsilon&\w{1}&\w{2}\\\w{3}&\varepsilon&\varepsilon
    \end{bmatrix}$
    \\\\
       $\begin{bmatrix}
          \varepsilon  & \w{1} & \varepsilon\\
        \w{1}  & \varepsilon & \varepsilon\\
          \varepsilon&\varepsilon& \w{1}
       \end{bmatrix}$  
       & 
       $\begin{bmatrix}
           \varepsilon &\w{1}\\
            \w{1} & \varepsilon
       \end{bmatrix}\shuffle\begin{bmatrix}
           \w{1} 
       \end{bmatrix}
       -
       \begin{bmatrix}
           \w{1} & \varepsilon & \varepsilon\\ \varepsilon& \varepsilon& \w{1} \\
          \varepsilon& \w{1}&\varepsilon
       \end{bmatrix}
       -2
       \begin{bmatrix}
           \varepsilon    &\varepsilon&\w{1}\\
           \w{1}&\varepsilon& \varepsilon  \\
          \varepsilon&\w{1}& \varepsilon
       \end{bmatrix}
       -2
       \begin{bmatrix}
           \varepsilon  &\w{1}  &\varepsilon\\
           \varepsilon& \varepsilon &\w{1} \\
          \w{1}& \varepsilon&\varepsilon
       \end{bmatrix}
              -3
       \begin{bmatrix}
            \varepsilon & \varepsilon&\w{1} \\ \varepsilon& \w{1}& \varepsilon \\
          \w{1}&\varepsilon& \varepsilon
       \end{bmatrix}$
    \end{tabular}
    \caption{Two-parameter shuffles}
    \label{tab:gens_in_shuffles}
    \end{small}
\end{table}

\subsection{Isomorphisms due to Lyndon generators}

\begin{corollary}\label{cor:isos_lynon}
    We have $(H,\blockShuffle,+)\cong(H,\shuffle,+)\cong(H,\qShuffle,+)$ as $\groundRing$-algebras. 
\end{corollary}

In \Cref{thm_iso_sh_qsh,thm:isos_Nikolas_Magic}  we lift this statement even to the level of Hopf algebras. 
Note that the explicit mappings constructed in  \Cref{cor:isos_lynon} are not compatible with $\Delta$, and thus, they cannot be considered as isomorphisms of Hopf algebras. 

\begin{proof}[Proof of \Cref{cor:isos_lynon}.]
With the universal property of $(H,\blockShuffle,+)$, we obtain an isomorphism of $\groundRing$-algebras $\Psi$  to  $(H,\shuffle,+)$ by setting $\Psi(\mathcal{B}_w):=\mathcal{B}_w$ for every $w\in \mathfrak{L}$. Its inverse is defined analogously via $\Psi^{-1}(\mathcal{B}_w)=\mathcal{B}_w$. 
    The remaining isomorphisms are constructed in the same way. 
\end{proof}

In the following, we illustrate one of the isomorphisms from \Cref{cor:isos_lynon}.

\begin{example}\label{ex:LyndonIsos}
Let $\Psi$ be the $\groundRing$-algebra isomorphism from $(H,\blockShuffle,+)$ to $(H,\shuffle,+)$. 
Lyndon generators are mapped to each other, e.g.
\begin{align*}\Psi(\begin{bmatrix}\w{1}\end{bmatrix})
&=\begin{bmatrix}\w{1}\end{bmatrix}
\quad\text{ or }\quad
  \Psi(
\begin{bmatrix}\w{1}&\varepsilon\\
    \varepsilon&\w{2}\end{bmatrix})=
    \begin{bmatrix}\w{1}&\varepsilon\\
    \varepsilon&\w{2}\end{bmatrix}.
    \end{align*}
    Otherwise we write compositions in terms of their generators and apply $\Psi$ such that it respects sums and products, e.g., 
    \begin{align*}
    \Psi(
\begin{bmatrix}\w{2}&\varepsilon\\
    \varepsilon&\w{1}\end{bmatrix})
    &=\Psi(
    \begin{bmatrix}\w{1}\end{bmatrix}\blockShuffle\begin{bmatrix}\w{2}\end{bmatrix}
    -\begin{bmatrix}\w{1}&\varepsilon\\
    \varepsilon&\w{2}\end{bmatrix}
    )
    =\Psi(
    \begin{bmatrix}\w{1}\end{bmatrix})\shuffle\Psi(\begin{bmatrix}\w{2}\end{bmatrix})
    -\Psi(\begin{bmatrix}\w{1}&\varepsilon\\
    \varepsilon&\w{2}\end{bmatrix}
    )
    \\
    &=
    \begin{bmatrix}\w{1}\end{bmatrix}\shuffle\begin{bmatrix}\w{2}\end{bmatrix}
    -\begin{bmatrix}\w{1}&\varepsilon\\
    \varepsilon&\w{2}\end{bmatrix}
    =
    \begin{bmatrix}\w{2}&\varepsilon\\
    \varepsilon&\w{1}\end{bmatrix}
    +
    \begin{bmatrix}\varepsilon&\w{2}\\
    \w{1}&\varepsilon\end{bmatrix}
    +
    \begin{bmatrix}\varepsilon&\w{1}\\
    \w{2}&\varepsilon\end{bmatrix}\end{align*}
    
    or 
        \begin{align*}
\Psi(\begin{bmatrix}\w{3}&\varepsilon&\varepsilon\\
    \varepsilon&\w{1}&\w{2}\end{bmatrix})
       &=\Psi(\begin{bmatrix}\w{1}&\w{2}\end{bmatrix}
       \blockShuffle\begin{bmatrix}\w{3}\end{bmatrix}
       -\begin{bmatrix}\w{1}&\w{2}&\varepsilon\\
    \varepsilon&\varepsilon&\w{3}\end{bmatrix})
    \\
    &=\Psi(\begin{bmatrix}\w{1}&\w{2}\end{bmatrix})
       \shuffle\Psi(\begin{bmatrix}\w{3}\end{bmatrix})
       -\Psi(\begin{bmatrix}\w{1}&\w{2}&\varepsilon\\
    \varepsilon&\varepsilon&\w{3}\end{bmatrix})\\
    &=\begin{bmatrix}\w{1}&\w{2}\end{bmatrix}
       \shuffle\begin{bmatrix}\w{3}\end{bmatrix}
       -\begin{bmatrix}\w{1}&\w{2}&\varepsilon\\
    \varepsilon&\varepsilon&\w{3}\end{bmatrix}\\
    &=
    \begin{bmatrix}\w{3}&\varepsilon&\varepsilon\\
    \varepsilon&\w{1}&\w{2}\end{bmatrix}
    +
    \begin{bmatrix}\w{1}&\varepsilon&\w{2}\\
    \varepsilon&\w{3}&\varepsilon\end{bmatrix}
+
\begin{bmatrix}\varepsilon&\varepsilon&\w{3}\\
    \w{1}&\w{2}&\varepsilon\end{bmatrix}
    +
    \begin{bmatrix}\varepsilon&\w{3}&\varepsilon\\
    \w{1}&\varepsilon&\w{2}\end{bmatrix}
    +
    \begin{bmatrix}
    \varepsilon&\w{1}&\w{2}\\\w{3}&\varepsilon&\varepsilon
    \end{bmatrix}.
    \end{align*}

\end{example}

However, an isomorphism of Hopf algebras is not obtained with this construction, as illustrated in the following.

\begin{example}
    We continue in the framework of \Cref{ex:LyndonIsos}, i.e.,  let $\Psi$ be the $\groundRing$-algebra isomorphism from $(H,\blockShuffle,+)$ to $(H,\shuffle,+)$.
    Then, with 
    $$\begin{bmatrix}
            \w{1}\end{bmatrix},
            \begin{bmatrix}
            \w{2}\end{bmatrix},
            \begin{bmatrix}
            \w{1}\end{bmatrix}\begin{bmatrix}
            \w{2}\end{bmatrix},
            \begin{bmatrix}
            \w{1}\end{bmatrix}\begin{bmatrix}
            \w{1}\end{bmatrix}\begin{bmatrix}
            \w{2}\end{bmatrix}\in\mathfrak{L}, 
            $$
 and  \Cref{eq:def_B_lyndon},  we have 
    \begin{align*}
        (\Psi\otimes\Psi)&\circ\Delta(\begin{bmatrix}
            \w{1}&\varepsilon&\varepsilon\\
            \varepsilon&\w{1}&\varepsilon\\
            \varepsilon&\varepsilon&\w{2}
        \end{bmatrix})\\
        &=\begin{bmatrix}
            \w{1}&\varepsilon&\varepsilon\\
            \varepsilon&\w{1}&\varepsilon\\
            \varepsilon&\varepsilon&\w{2}
        \end{bmatrix}\otimes\ec
        +\begin{bmatrix}
            \w{1}\end{bmatrix}\otimes\begin{bmatrix}
            \w{1}&\varepsilon\\
            \varepsilon&\w{2}
        \end{bmatrix}
        +\Psi(\begin{bmatrix}
            \w{1}&\varepsilon\\
            \varepsilon&\w{1}
        \end{bmatrix})\otimes\begin{bmatrix}
            \w{2}
        \end{bmatrix}
        +\ec\otimes\begin{bmatrix}
            \w{1}&\varepsilon&\varepsilon\\
            \varepsilon&\w{1}&\varepsilon\\
            \varepsilon&\varepsilon&\w{2}
        \end{bmatrix}
        \\
        &\not=\Delta(\begin{bmatrix}
            \w{1}&\varepsilon&\varepsilon\\
            \varepsilon&\w{1}&\varepsilon\\
            \varepsilon&\varepsilon&\w{2}
            \end{bmatrix})
            =\Delta\circ\Psi(\begin{bmatrix}
            \w{1}&\varepsilon&\varepsilon\\
            \varepsilon&\w{1}&\varepsilon\\
            \varepsilon&\varepsilon&\w{2}
        \end{bmatrix}), 
    \end{align*}
    since with \Cref{tab:gens_in_block_shuffles}, \begin{align*}\Psi(\begin{bmatrix}
            \w{1}&\varepsilon\\
            \varepsilon&\w{1}
        \end{bmatrix})&=\Psi(\frac12\begin{bmatrix}
           \w{1}
       \end{bmatrix}\blockShuffle\begin{bmatrix}
           \w{1}
       \end{bmatrix})\\&=\frac12\begin{bmatrix}
           \w{1}
       \end{bmatrix}\shuffle\begin{bmatrix}
           \w{1}
       \end{bmatrix}=\begin{bmatrix}
            \w{1}&\varepsilon\\
            \varepsilon&\w{1}
       \end{bmatrix}+\begin{bmatrix}
            \varepsilon&\w{1}\\
            \w{1}&\varepsilon
       \end{bmatrix}\not=\begin{bmatrix}
            \w{1}&\varepsilon\\
            \varepsilon&\w{1}
        \end{bmatrix}
        \end{align*}
        so $\Psi$ does not respect the coproduct $\Delta$. 
\end{example}

\section{Hoffman's isomorphism for two parameters}
\label{sec:hoffman}

In this section we introduce a novel two-parameter version of Hoffman's isomorphism  (\Cref{def:Phi}) together with its explicit inverse (\Cref{lem:Phi_inv}). 
With this we obtain the following main result.

\begin{theorem}\label{thm_iso_sh_qsh}
The Hopf algebras
$(H,\shuffle,\Delta)$ and $(H,\qShuffle,\Delta)$ are isomorphic. 
\end{theorem}

We prove this in this section. A sequence $I=(I_1,\dots,I_\ell)\in\N^\ell$ of positive integers is called a (classical, one-parameter) \DEF{composition} of $n\in\N_0$, if $I_1+\dots+I_\ell=n$. Conventionally the empty composition is counted as the sole composition of $0$. 
Denote the set of all compositions of $n$ by $\mathcal C(n)$.
\index[general]{C@$\mathcal{C}$}
\index[general]{len@$\len:\mathcal{C}(n)\rightarrow\N_0$}

For nonempty $\mathbf{a}\in\composition$, $I\in\mathcal{C}(\rows(\mathbf{a}))$ and $J\in\mathcal{C}(\cols(\mathbf{a}))$ let $\mathbf{a}_{I,J}$ denote the action of $I$ and $J$ on the rows and columns of $\mathbf{a}$, respectively.  
 This action is formally provided in \Cref{def:matrix_action}.
Let $\len(I)=\ell$ denote the \DEF{length} and $I!=I_1!\dots I_{\len(I)}!$ the \DEF{factorial}  of $I\in\N^\ell$.  

The following explicit mapping $\Phi$ is motivated by Hoffman's isomorphism \cite[Thm.~2.5]{H99}, and here generalized for two parameters.

\begin{definition}\label{def:Phi}
Let $\Phi$ be the $\groundRing$-linear mapping from  $(H,\shuffle,\Delta)$ to $(H,\qShuffle,\Delta)$
which satisfies $\Phi(\ec):=\ec$ and 
\[\Phi(\mathbf{a})
:=\sum_{\substack{I\in\mathcal{C}(\rows(\mathbf{a}))\\J\in\mathcal{C}(\cols(\mathbf{a}))}}\frac{1}{I!J!}\,\mathbf{a}_{I,J}\]
for every nonempty $\groundRing$-basis element $\mathbf{a}\in\composition$.
\end{definition}
\index[general]{Phi@$\Phi$}

\begin{example}
  \begin{align*}  
 \Phi(\begin{bmatrix}\w{1}\end{bmatrix})&=\begin{bmatrix}\w{1}\end{bmatrix},
 \quad\quad
\Phi(\begin{bmatrix}\w{1}\\\w{2}\end{bmatrix})=\begin{bmatrix}\w{1}\\\w{2}\end{bmatrix}+\frac12\begin{bmatrix}\w{1}\star\w{2}\end{bmatrix},
    \\
\Phi(\begin{bmatrix}\w{1}&\varepsilon\\\varepsilon&\w{2}\end{bmatrix})
&=\begin{bmatrix}\w{1}&\varepsilon\\\varepsilon&\w{2}\end{bmatrix}
+\frac12\begin{bmatrix}\w{1}\\\w{2}\end{bmatrix}
+\frac12\begin{bmatrix}\w{1}&\w{2}\end{bmatrix}
+\frac14\begin{bmatrix}\w{1}\star\w{2}\end{bmatrix},
   \\
\Phi(\begin{bmatrix}\w{1}&\varepsilon\\\varepsilon&\w{2}\\\varepsilon&\w{3}\end{bmatrix})
&=\begin{bmatrix}\w{1}&\varepsilon\\\varepsilon&\w{2}\\\varepsilon&\w{3}\end{bmatrix}
+\frac12\begin{bmatrix}\w{1}&\w{2}\\\varepsilon&\w{3}\end{bmatrix}
+\frac12\begin{bmatrix}\w{1}&\varepsilon\\\varepsilon&\w{2}\star\w{3}\end{bmatrix}
+\frac12\begin{bmatrix}\w{1}\\\w{2}\\\w{3}\end{bmatrix}\\
&\;\;\;\;\;\;+\frac14\begin{bmatrix}\w{1}\star\w{2}\\\w{3}\end{bmatrix}
+\frac14\begin{bmatrix}\w{1}\\\w{2}\star\w{3}\end{bmatrix}
+\frac1{12}\begin{bmatrix}\w{1}\star\w{2}\star\w{3}\end{bmatrix}
     \end{align*}
\end{example}

We first show that $\Phi$ is multiplicative. 

\begin{lemma}\label{lem:Phi_hom} 
For all  $\mathbf{a},\mathbf{b}\in\composition$, 
\[\Phi(\mathbf{a}\shuffle\mathbf{b})=\Phi(\mathbf{a})\qShuffle\Phi(\mathbf{b}).\]
\end{lemma}


\subsection{Row and column action on matrix compositions}
In order to prove \Cref{lem:Phi_hom}, we rewrite 
\Cref{def:Phi} using matrix notation. For this, we introduce a one-hot matrix encoding of a (classical, one-parameter) composition. 
\begin{lemma}\label{lem:iso_brew_compositions}
     For all $v,\ell\in\N$, the set  \[\brew(v,\ell):=\left\{\begin{bmatrix}\e{\iota_1}&\dots&\e{\iota_\ell}\end{bmatrix}\in\N_0^{v\times \ell}\mid \iota_1\leq \dots\leq\iota_\ell\text{ and }\iota_j-\iota_{j-1}\leq 1\right\}\]
     is in one-to-one correspondence to the set $\left\{I\in\mathcal{C}(\ell)\mid \len (I)=v\right\}$  of compositions with  length $v$. 
\end{lemma}
\index[general]{fuse@$\brew$}

\begin{example}The one-hot encodings of all $\mathcal{C}(4)$ with length $2$ are given by
    \[\brew(2,4)=\left\{\begin{bmatrix}
                    1&1&1&0\\
                    0&0&0&1\end{bmatrix}, 
                    \begin{bmatrix}
                    1&1&0&0\\
                    0&0&1&1\end{bmatrix}, 
                    \begin{bmatrix}
                    1&0&0&0\\
                    0&1&1&1\end{bmatrix}\right\}.\]
\end{example}

\index[general]{aei@$a.\e{i}$ with $a,i\in\N$}
\begin{proof}[Proof of \Cref{lem:iso_brew_compositions}.]
    For any $a\in\N$  let 
    \begin{equation}\label{def:a_Dot_ej}
        a.\e{i}:=\begin{bmatrix}\e{i}&\dots&\e{i}\end{bmatrix}\in\N_0^{v\times a}.
    \end{equation}
    With this  we obtain a bijective mapping via  
    \[\begin{bmatrix}
I_1.\e{1}&\dots&I_v.\e{v}
    \end{bmatrix}\in\brew(v,\ell)\]
     for every $I\in\mathcal{C}(\ell)$ with $\len(I)=v$.  
\end{proof}

\begin{definition}\label{def:matrix_action}
  For all nonempty  $\mathbf{a}\in\composition$ and  \[(I,J)\in\mathcal{C}(\rows(\mathbf{a}))\times\mathcal{C}(\cols(\mathbf{a}))\] let \[(\mathbf{P},\mathbf{Q})\in\brew(\len( I),\rows(\mathbf{a}))\times\brew(\len( J),\cols(\mathbf{a}))\] denote the corresponding one-hot encodings due to \Cref{lem:iso_brew_compositions}.
   Then, we set  \[\mathbf{a}_{I,J}:=\mathbf{P}\mathbf{a}\mathbf{Q}\]
    using the matrix action of $\mathbf{P}$ and $\mathbf{Q}$ on the rows and columns of $\mathbf{a}$, respectively. 
    For $\mathbf{P}$ let 
 $\mathbf{P}!:=I!$ denote the \DEF{factorial} of its corresponding composition $I$.
\end{definition}
The one-hot encodings of compositions satisfy the following properties.

\begin{lemma}\label{lem:diag_of_brew_multiplic}
    For all $\ell=\ell_1+\ell_2$ and $v=v_1+v_2$ we have \[
    \left\{\diag(\mathbf{U},\mathbf{V})\in\brew(\ell,v)\;\begin{array}{|l}
     \size(\mathbf{U})=(\ell_1,v_1)\\\size(\mathbf{V})=(\ell_2,v_2)\end{array}\!\right\}=\diag(\brew(\ell_1,v_1),\brew(\ell_2,v_2)),\] 
    where the operation $\diag$ is applied entrywise. 
\end{lemma}

\begin{lemma}\label{lem:factorial_of_brew_multiplic}
    If $\mathbf{V}=\diag(\mathbf{V}_1,\mathbf{V}_2)\in\brew(\ell,v)$ with $\mathbf{V}_i\in\brew(\ell_i,v_i)$ for $i\in\{1,2\}$,  then $\mathbf{V}!=\mathbf{V}_1!\mathbf{V}_2!$. 
\end{lemma}
Furthermore,  all matrices in $\brew(v,\ell)$ are right invertible, and \Cref{def:Phi} translates to 
\begin{equation}\Phi(\mathbf{a})
=\sum_{u,v\in\N}\,\sum_{\substack{\mathbf{U}\in \brew(u,\rows(\mathbf{a}))\\\mathbf{V}\in \brew(v,\cols(\mathbf{a}))}}\frac{1}{\mathbf{U}!\mathbf{V}!}\,\mathbf{U}\mathbf{a}\mathbf{V}^\top.\end{equation}

\subsection{Multiplication}
From the proof of \cite[Thm.~2.5]{H99} we recall the following combinatorial (one-parameter) argument in the language of our one-hot matrix encodings. 
We apply this argument on both axes simultaneously, and thus show that $\Phi$ is multiplicative (\Cref{lem:Phi_hom}).
\begin{lemma}\label{lem:technical_hoffman1param}
Let $a,b\in\N$ be fixed. 
\begin{enumerate}
    \item\label{lem:technical_hoffman1param1}
    For all $u_1,u_2,j\in\N$, the mapping 
    \begin{align*}M_{u_1,u_2,j}:\QSH(u_1,u_2;j)\times\brew(u_1,a)\times\brew(u_2,b)&\rightarrow\N_0^{j\times (a+b)}\\
    (\mathbf{A},\mathbf{U}_1,\mathbf{U}_2)&\mapsto \mathbf{A}\diag(\mathbf{U}_1,\mathbf{U}_2)
    \end{align*}
    is injective. 
    \item\label{lem:technical_hoffman1param2}
    For all $j\in\N$ and  $\mathbf{W}\in\Ima(N_j)$ with 
    \begin{align*}N_{j}:\brew(j,a+b)\times\SH(a,b)&\rightarrow\N_0^{j\times (a+b)}\\
    (\mathbf{U},\mathbf{P})&\mapsto \mathbf{U}\mathbf{P}
    \end{align*}
    there is a unique $\mathbf{U}$ such that $N_j(\mathbf{U},\mathbf{P})=\mathbf{W}$. 
    \item\label{lem:technical_hoffman1param3} For all $j\in\N$, 
    \[\Ima(N_j)=\biguplus_{u_1,u_2}\Ima(M_{u_1,u_2,j}).\]
    \item\label{lem:technical_hoffman1param4}
    Each non-empty preimage $N_j^{-1}(\mathbf{UP})$
     has the cardinality 
    \[\# N_j^{-1}(\mathbf{UP})=\frac{\mathbf{U}!}{\mathbf{U}_1!\mathbf{U}_2!}\]
    where $\mathbf{UP}=\mathbf{A}\diag(\mathbf{U}_1,\mathbf{U}_2)$ due to part \ref{lem:technical_hoffman1param3}.

    \end{enumerate}
\end{lemma}
Despite this being well-known, we provide a proof for completeness.
\begin{proof}[Proof of \Cref{lem:technical_hoffman1param}.]
    The elements in $\Ima(M_{u_1,u_2,j})$ are given by actions on the columns in $\mathbf{A}$, that is 
    \begin{align}
    &\mathbf{A}\diag(\mathbf{U}_1,\mathbf{U}_2)\nonumber\\
    &=\begin{bmatrix}\e{\iota_1}&\dots&\e{\iota_{u_1}}&\e{\kappa_1}&\dots&\e{\kappa_{u_2}}\end{bmatrix}
    \begin{bmatrix}I_1.e_1&\dots&I_{u_1}.\e{u_1}&0&\dots&0\\
    0&\dots&0&J_1.e_1&\dots&J_{u_2}.\e{u_2}\end{bmatrix}\nonumber\\
    &=\begin{bmatrix}I_1.\e{\iota_1}&\dots&I_{u_1}.\e{\iota_{u_1}}&J_1.\e{\kappa_1}&\dots&J_{u_2}.\e{\kappa_{u_2}}\end{bmatrix}\label{eq:AdiagU1U2}
    \end{align}
    where $t.\e{i}$ is according to \cref{def:a_Dot_ej} for $t\in\N$.
    Whenever  
    \begin{align*}&\begin{bmatrix}I_1.\e{\iota_1}&\dots&I_{u_1}.\e{\iota_{u_1}}&J_1.\e{\kappa_1}&\dots&J_{u_2}.\e{\kappa_{u_2}}\end{bmatrix}\\
    &=
    \begin{bmatrix}I'_1.\e{\iota'_1}&\dots&I'_{u_1}.\e{\iota'_{u_1}}&J'_1.\e{\kappa'_1}&\dots&J'_{u_2}.\e{\kappa'_{u_2}}\end{bmatrix}
    \end{align*}
    then $I=I'$, $J=J'$, $\iota=\iota'$ and $\kappa=\kappa'$, i.e., part \ref{lem:technical_hoffman1param1}.
    
For part \ref{lem:technical_hoffman1param2}.~let $\mathbf{UP}=\mathbf{U}'\mathbf{P}'$ with $\mathbf{U},\mathbf{U}'\in\brew(j,a+b)$ and $\mathbf{P},\mathbf{P}'\in\SH(a,b)$. Then  $\mathbf{U}=\mathbf{U}'\mathbf{P}'\mathbf{P}^{-1}$, and therefore $\mathbf{U}=\mathbf{U}'$ due to $\mathbf{P}'\mathbf{P}^{-1}\in\Sigma_{a+b}$. 

For part \ref{lem:technical_hoffman1param3}.~the elements $\mathbf{U}\mathbf{P}\in\Ima(N_j)$
    are given by $\mathbf{U}$, acting on the rows in
    \[
    \mathbf{P}=\begin{bmatrix}
      \e{\sigma_1}&\dots&\e{\sigma_{a+b}}
  \end{bmatrix}
  =\begin{bmatrix}
      \e{\mu_1}&\dots&\e{\mu_{a}}&\e{\nu_1}&\dots&\e{\nu_{b}}
  \end{bmatrix}\]
 with permutation $\sigma\in\Sigma_{a+b}$ and two strictly increasing and disjoint $\mu$ and $\nu$. 
  This determines unique $K\in\mathcal{C}(j,a+b)$  and $m,n\in\N_0^j$ such that
  \begin{align}\label{eq:UP_proof_coices}
  \mathbf{U}\mathbf{P}&=
  \begin{bmatrix}
      K_1.\e{1}&\dots&K_j.\e{j}
  \end{bmatrix}
  \begin{bmatrix}
      \e{\sigma_1}&\dots&\e{\sigma_{a+b}}
  \end{bmatrix}\\
  &=\begin{bmatrix}m_1.\e{1}&\dots&m_{j}.\e{j}&n_1.\e{1}&\dots&n_{j}.\e{j}\nonumber\end{bmatrix}
\end{align}
   with $0.e_i$ encoding an empty block whenever $m_i$ or $n_i$ is zero for $1\leq i\leq j$. 
   We obtain \cref{eq:AdiagU1U2} by viewing $m$ and $n$ as elements $I\in\mathcal{C}(u_1,a)$ and $J\in\mathcal{C}(u_2,b)$, i.e., by omitting all the $j-u_1$ zero entries in $m$, and the $j-u_2$ zero entries in $n$. This implies 
   \begin{equation}\label{eq:proof_NM_same_Ima}
     \Ima(N_j)\subseteq\bigcup_{u_1,u_2}\Ima(M_{u_1,u_2,j}).  
   \end{equation}
   In \cref{eq:AdiagU1U2} we can count the number of non-zero rows $u_1$ and $u_2$ in the first $a$ and remaining $b$ columns, which implies that the union in \cref{eq:proof_NM_same_Ima} is disjoint. In part \ref{lem:technical_hoffman1param2}.~we showed already that the only freedom for   \cref{eq:UP_proof_coices}  comes from different choices of $\mathbf{P}$. 
   In particular there are exactly
   \[\frac{K_1!\dots K_j!}{m_1!\dots m_j!n_1!\dots n_j!}=\frac{K_1!\dots K_j!}{I_1!\dots I_{\len(I)}!J_1!\dots J_{\len(J)}!}=\frac{\mathbf{U}!}{\mathbf{U}_1!\mathbf{U}_2!}\]
   choices of $\mathbf{P}'\in\SH(a,b)$ such that $\mathbf{UP}=\mathbf{U}\mathbf{P}'$ is satisfied. 
   It remains to show equality in \cref{eq:proof_NM_same_Ima}.
   For this let 
   \[\mathbf{W}\in\biguplus_{u_1,u_2}\Ima(M_{u_1,u_2,j}),\]
   so there are unique $u_1$ and $u_2$ such that $\mathbf{W}\in\Ima(M_{u_1,u_2,j})$. With part \ref{lem:technical_hoffman1param1}.~there is a unique \[(\mathbf{A},\mathbf{U}_1,\mathbf{U}_2)\in\QSH(u_1,u_2;j)\times\brew(u_1,a)\times\brew(u_2,b)\]
   such that $\mathbf{W}=\mathbf{A}\diag(\mathbf{U}_1,\mathbf{U}_2)$. 
   Let $\iota,\kappa,I$ and $J$ as in \cref{eq:AdiagU1U2}. 
   Since $\mathbf{A}$ is right invertible there exist disjoint $\tilde\iota_1<\dots<\tilde\iota_{a}$ and $\tilde\kappa_1<\dots<\tilde\kappa_{b}$ such that 
   $\{\tilde\iota_1,\dots,\tilde\iota_{a}\}=\{\iota_1,\dots,\iota_{u_1}\}$, 
   $\{\tilde\kappa_1,\dots,\tilde\kappa_b\}=\{\kappa_1,\dots,\kappa_{u_1}\}$
   and 
   $$\{\tilde\iota_1,\dots,\tilde\iota_{a}\}\cup\{\tilde\kappa_1,\dots,\tilde\kappa_b\}
   =\{\iota_1,\dots,\iota_{u_1}\}\cup\{\kappa_1,\dots,\kappa_{u_2}\}
   =\{1,\dots,a+b\}.$$
   This defines 
   \[
    \mathbf{P}:=
  \begin{bmatrix} \e{\tilde\iota_1}&\dots&\e{\tilde\iota_{a}}&\e{\tilde\kappa_1}&\dots&\e{\tilde\kappa_{b}}
  \end{bmatrix}\in\SH(a,b).\]
  Choose now $\mathbf{U}\in\brew(j,a+b)$ such that $N_j(\mathbf{U},\mathbf{P})=\mathbf{U}\mathbf{P}=\mathbf{W}$. 
\end{proof}

With this we verify that $\Phi$ is multiplicative.

\begin{proof}[Proof of \Cref{lem:Phi_hom}]
For all nonempty matrix compositions $\mathbf{a}$ and $\mathbf{b}$, 
\begin{align*}
\Phi(\mathbf{a}\shuffle\mathbf{b})
&=\sum_{\substack{j\\k}}\sum_{\substack{\mathbf{U}\in \brew(j,\rows(\diag(\mathbf{a},\mathbf{b})))\\\mathbf{V}\in \brew(k,\cols(\diag(\mathbf{a},\mathbf{b})))\\\mathbf{P}\in \SH(\rows(\mathbf{a}),\rows(\mathbf{b}))\\\mathbf{Q}\in\SH(\cols(\mathbf{a}),\cols(\mathbf{b}))}}\frac{1}{\mathbf{U}!\mathbf{V}!}\,\mathbf{U}\mathbf{P}\diag(\mathbf{a},\mathbf{b})\mathbf{Q}^\top\mathbf{V}^\top\\
&\overset{\ref{lem:technical_hoffman1param}.\ref{lem:technical_hoffman1param4}}=\,\sum_{\substack{j\\k}}
\sum_{\substack{\mathbf{UP}\in\Ima(N_j)\\
\mathbf{VQ}\in\Ima(N_k)}}
\frac{\#N_j^{-1}(\mathbf{UP})\#N_k^{-1}(\mathbf{VQ})}{\mathbf{U}!\mathbf{V}!}\mathbf{U}\mathbf{P}\diag(\mathbf{a},\mathbf{b})\mathbf{Q}^\top\mathbf{V}^\top
\\
&\overset{(\ref{eq:summands_Phi})}=\sum_{\substack{j\\k\\u_1\\u_2\\v_1\\v_2}}
\sum_{\substack{\mathbf{U}_1\in \brew(u_1,\rows(\mathbf{a}
))\\
\mathbf{U}_2\brew(u_2,\rows(\mathbf{b}))\\
\mathbf{V}_1\in \brew(v_1,\cols(\mathbf{a}))\\
\mathbf{V}_2\in \brew(v_2,\cols(\mathbf{b}))\\
\mathbf{A}\in \QSH(u_1,u_2;j)\\
\mathbf{B}\in\QSH(v_1,v_2;k)}}\frac{1}{\mathbf{U}_1!\mathbf{U}_2!\mathbf{V}_1!\mathbf{V}_2!}\mathbf{A}\diag(\mathbf{U}_1\mathbf{a}\mathbf{V}_1^\top,\mathbf{U}_2\mathbf{b}\mathbf{V}_2^\top)\mathbf{B}^\top\\
&=\Phi(\mathbf{a})\qShuffle\Phi(\mathbf{b}),
\end{align*} 
with summands 
\begin{align}
    \mathbf{U}\mathbf{P}\diag(\mathbf{a},\mathbf{b})\mathbf{Q}^\top\mathbf{V}^\top\nonumber
    &=M(\mathbf{A},\mathbf{U}_1,\mathbf{U}_2)\diag(\mathbf{a},\mathbf{b})M(\mathbf{B},\mathbf{V}_1,\mathbf{V}_2)^\top\nonumber\\&=
    \mathbf{A}\diag(\mathbf{U}_1,\mathbf{U}_2)\diag(\mathbf{a},\mathbf{b}){\diag(\mathbf{V}_1,\mathbf{V}_2)}^\top\mathbf{B}^\top\nonumber\\
    &=\mathbf{A}\diag(\mathbf{U}_1\mathbf{a}\mathbf{V}_1^\top,\mathbf{U}_2\mathbf{b}\mathbf{V}_2^\top)\mathbf{B}^\top\label{eq:summands_Phi},
    \end{align}
with \Cref{lem:technical_hoffman1param}.\ref{lem:technical_hoffman1param3}, and due to elementary block arithmetic. 
\end{proof}

\subsection{Invertibility}
We provide the explicit inverse of our two-parameter Hoffman's isomorphism. 
\begin{lemma}\label{lem:Phi_inv} 
  The homomorphism $\Phi$ is invertible, where 
    \[\Phi^{-1}(\mathbf{a}):=\sum_{\substack{I\in\mathcal{C}(\rows(\mathbf{a}))\\J\in\mathcal{C}(\cols(\mathbf{a}))}}\frac{(-1)^{\rows(\mathbf{a})-\len(I)+\cols(\mathbf{a})-\len(J)}}{I_1\dots I_{\len(I)}J_1\dots J_{\len(J)} }\,\mathbf{a}_{I,J}\]
    for every nonempty $\mathbf{a}\in\composition$. 
\end{lemma}

\begin{example}
    \begin{align*}
        \Phi^{-1}(\begin{bmatrix}\w{1}\end{bmatrix})&=\begin{bmatrix}\w{1}\end{bmatrix},\quad\quad
\Phi^{-1}(\begin{bmatrix}\w{1}\\\w{2}\end{bmatrix})=\begin{bmatrix}\w{1}\\\w{2}\end{bmatrix}-\frac12\begin{bmatrix}\w{1}\star\w{2}\end{bmatrix},
    \\
\Phi^{-1}(\begin{bmatrix}\w{1}&\varepsilon\\\varepsilon&\w{2}\end{bmatrix})
&=\begin{bmatrix}\w{1}&\varepsilon\\\varepsilon&\w{2}\end{bmatrix}
-\frac12\begin{bmatrix}\w{1}\\\w{2}\end{bmatrix}
-\frac12\begin{bmatrix}\w{1}&\w{2}\end{bmatrix}
+\frac14\begin{bmatrix}\w{1}\star\w{2}\end{bmatrix}
\end{align*}
    \end{example}

In order to prove this, we introduce further notation. 
For two fixed formal power series in one variable $t$, and without constant term,  $$f=\sum_{i\in\N}f_it^i\,,\;g=\sum_{i\in\N}g_it^i\in t\cdot\groundRing\llbracket t\rrbracket,$$ 
we define a $\groundRing$-linear \DEF{evaluation} 
$\evaluate_{f,g}\in\End(H)$ via 
\[\evaluate_{f,g}(\mathbf{a}):=\sum_{\substack{I\in\mathcal{C}(\rows(\mathbf{a}))\\J\in\mathcal{C}(\cols(\mathbf{a}))}}f_{I_1}\dots f_{I_{\len(I)}}g_{J_1}\dots g_{J_{\len( J)}}\mathbf{a}_{I,J}\]
for every $\groundRing$-basis element $\mathbf{a}\in\composition$. 
For every $i\in\N_0$ let $[t^i]f:=f_i$ denote the coefficient of $f$ at monomial $t^i$. 
We introduce the two-parameter generalization of \cite[Thm.~3.1]{Hoffman_2017}.\index[general]{eval@$\evaluate$}

\begin{lemma}\label{lem_HI17_thm31_2param} We have 
    \[\evaluate_{f,p}\circ\evaluate_{g,q}=\evaluate_{f\circ g,p\circ q}\]
    for every $f,g,p,q\in t\cdot\groundRing\llbracket t\rrbracket$.
\end{lemma}

We recall the one-parameter argument from \cite[eq.~(9)]{Hoffman_2017},
\begin{equation}\label{eq:inv_hoff_1}
    [t^k](f\circ g)=\sum_{j=1}^k[t^j]f[t^k]g^j
\end{equation}
for every $f,g\in t\cdot\groundRing\llbracket t\rrbracket$ and $k\in\N$.
Let $\mathbf{a}\in\composition$ be fixed.
For $I\in\mathcal{C}(\rows(\mathbf{a}))$ with $\len(I)=x$  and $J\in\mathcal{C}(x)$ let 
\[J\circ I:=(I_1+\dots +I_{J_1},I_{J_1+1}+\dots+I_{J_1+J_2},\dots,I_{J_1+\dots+J_{\ell-1}+1}+I_x)\in\mathcal{C}(\rows(\mathbf{a}))\]
denote the concatenation of $I$ and $J$. Note that $\len(J\circ I)=\len(J)$. Recall further from \cite[eq.~(11)]{Hoffman_2017} that for all $K\in\mathcal{C}(\rows(\mathbf{a}))$, 
\begin{align}\label{eq:Hoffman17_eq11}
    [t^{K_1}]&(f\circ g)\dots [t^{K_{\len(K)}}](f\circ g)\nonumber\\
    &=\sum_{x=1}^{\rows(\mathbf{a})}\sum_{\substack{J\in\mathcal{C}(x)\\\len(J)=\len( K)}}\sum_{\substack{I\in\mathcal{C}(\rows(\mathbf{a}))\\J\circ I=K}}[t^{J_1}]f\dots[t^{J_{\len( K)}}]f[t^{I_1}]g\dots[t^{I_{x}}]g.
\end{align}
 
\begin{proof}[Proof of \Cref{lem_HI17_thm31_2param}.]
We use the recalled one-parameter argument from above on rows and columns, i.e. 
\begin{align*}
    &\evaluate_{f\circ g,p\circ q}(\mathbf{a})\\
    &=\sum_{\substack{K\in\mathcal{C}(\rows(\mathbf{a}))\\L\in\mathcal{C}(\cols(\mathbf{a}))}}[t^{K_1}](f\circ g)\dots [t^{K_{\lvert K\rvert}}](f\circ g)[t^{L_1}](p\circ q)\dots t^{L_{\lvert L\rvert)}}(p\circ q)\mathbf{a}_{K,L}\\
    &\overset{(\ref{eq:Hoffman17_eq11})}=
    \sum_{\substack{K\in\mathcal{C}(\rows(\mathbf{a}))\\L\in\mathcal{C}(\cols(\mathbf{a}))\\1\leq x\leq\rows(\mathbf{a})\\
    1\leq y\leq \cols(\mathbf{a})}}
    \sum_{\substack{J\in\mathcal{C}(x)\\B\in\mathcal{C}(y)\\\len(J)=\len(K)\\\len(B)=\len(L)}}
    \sum_{\substack{I\in\mathcal{C}(\rows(\mathbf{a}))\\A\in\mathcal{C}(\cols(\mathbf{a}))\\\text{with}\\J\circ I=K\\
    B\circ A=L}}
    \left(\prod_{a}[t^{I_a}]g\right)
    \left(\prod_{b}[t^{A_b}]q\right)
    \left(\prod_{s}[t^{J_s}]f\right)
    \left(\prod_{r}[t^{B_r}]p\right)\,\mathbf{a}_{K,L}
    \\
    &=\sum_{\substack{I\in\mathcal{C}(\rows(\mathbf{a}))\\A\in\mathcal{C}(\cols(\mathbf{a}))}}\left(\prod_{a}[t^{I_a}]g\right)
    \left(\prod_{b}[t^{A_b}]q\right)\left(
    \sum_{\substack{1\leq x\leq \rows(\mathbf{a})\\
    1\leq y\leq \cols(\mathbf{a})}}
    \sum_{\substack{J\in\mathcal{C}(x)\\B\in\mathcal{C}(y)\\\text{with}\\\len(I)=x\\
    \len(A)=y}}\left(\prod_{s}[t^{J_s}]f\right)
    \left(\prod_{r}[t^{B_r}]p\right){\left(\mathbf{a}_{I,A}\right)}_{J,B}\right)
    \\
    &=
    \sum_{\substack{I\in\mathcal{C}(\rows(\mathbf{a}))\\A\in\mathcal{C}(\cols(\mathbf{a}))}}[t^{I_1}]g\dots[t^{I_{\len( I)}}]g
    [t^{A_1}]q\dots[t^{A_{\len(A)}}]q\,\evaluate_{f,p}
    \left(\mathbf{a}_{I,A}\right)\\
    &=\evaluate_{f,p}\circ\evaluate_{g,q}(\mathbf{a})
\end{align*}
for every $\groundRing$-basis element $\mathbf{a}\in\composition$. 
\end{proof}

\begin{proof}[Proof of \Cref{lem:Phi_inv}]
    This follows by choosing $f(t):=g(t) := e^t-1$ in the situation of \Cref{lem_HI17_thm31_2param}.
\end{proof}
\subsection{Comultiplication}

\begin{lemma}\label{lem:Phi_hom_coprod} 
The homomorphism $\Phi$ respects deconcatenation, i.e., 
\[(\Phi\otimes\Phi)\circ\Delta=\Delta\circ\Phi.\]
\end{lemma}
In order to prove \Cref{lem:Phi_hom_coprod} we introduce splittings similar to \cite[Sec.~5.3]{DS23}.
\begin{definition}
For fixed $\mathbf{a}=\diag(\mathbf{u}_1,\dots,\mathbf{u}_{\len(\mathbf{a})})$ and $0<\alpha<\len(\mathbf{a})$, we say that  
\[(\mathbf{U},\mathbf{V})\in\brew(u,\rows(\mathbf{a}))\times\brew(v,\cols(\mathbf{a}))\] is \DEF{$\alpha$-decomposable}, if 
\begin{equation}\label{eq:def_alpha_dec}(\mathbf{U},\mathbf{V})
=\left(\diag(
    \mathbf{U}_{1},\mathbf{U}_{2}),
\diag(
    \mathbf{V}_{1},\mathbf{V}_{2})\right)\end{equation}
with $u=u_1+u_2$ and $v=v_1+v_2$ such that 
\[(\mathbf{U}_{1},\mathbf{U}_{2})
\in
\brew\left(u_1,\sum_{1\leq i\leq\alpha}\rows(\mathbf{u}_i)\right)
\times
\brew\left(u_2,\sum_{\alpha+1\leq i\leq \len(\mathbf{a})}\rows(\mathbf{u}_i)\right)\]
and 
\[(\mathbf{V}_{1},\mathbf{V}_{2})
\in
\brew\left(v_1,\sum_{1\leq i\leq\alpha}\cols(\mathbf{u}_i)\right)
\times
\brew\left(v_2,\sum_{\alpha+1\leq i\leq \len(\mathbf{a})}\cols(\mathbf{u}_i)\right).\]
\end{definition}
\begin{lemma}
    If $(\mathbf{U},\mathbf{V})$ is $\alpha$-decomposable, then its decomposition from \cref{eq:def_alpha_dec} is unique. 
\end{lemma}
Similar to \cite[Lem.~5.21]{DS23} we describe the block structure of matrix compositions after the action due to \Cref{def:matrix_action}. 
\begin{lemma}\label{lem:splittings}
   For fixed  $\mathbf{U}$,$\mathbf{V}$ and $\mathbf{a}$, the set 
   \[\{(\mathbf{x},\mathbf{y})\in\composition^2\mid\mathbf{x}\not=\ec\not=\mathbf{y}\land\diag(\mathbf{x},\mathbf{y})=\mathbf{U}\mathbf{a}\mathbf{V}^\top\}\]
   is in one-to-one correspondence to all \DEF{splittings}, 
   \[\mathfrak{S}_{\mathbf{U},\mathbf{V}}^{\mathbf{a}}:=\{1\leq\alpha<\len(\mathbf{a})\mid\;(\mathbf{U},\mathbf{V})\text{ is }\alpha\text{-decomposable}\}.\]
\end{lemma}
\index[general]{S@$\mathfrak{S}_{\mathbf{U},\mathbf{V}}^{\mathbf{a}}$}
\begin{proof}
If $\alpha\in\mathfrak{S}_{\mathbf{U},\mathbf{V}}^{\mathbf{a}}$ then  $(\mathbf{U},\mathbf{V})$ is $\alpha$-decomposable, and therefore,
\begin{align}\label{eq:UaV_alpha_decomp_diag_form}
    \mathbf{U}\mathbf{a}\mathbf{V}^\top
    &=\diag(\mathbf{U}_1,\mathbf{U}_2)\,\diag(\mathbf{u}_1,\dots,\mathbf{u}_{\len(\mathbf{a})})\,\diag(\mathbf{V}_1,\mathbf{V}_2)\nonumber\\
&=\diag(\mathbf{U}_1\diag(\mathbf{u}_1,\dots,\mathbf{u}_\alpha)\mathbf{V}_1,\mathbf{U}_2\diag(\mathbf{u}_{\alpha+1},\dots,\mathbf{u}_{\len(\mathbf{a})})\mathbf{V}_2)\\
&=:\diag(\mathbf{x}^{(\alpha)},\mathbf{y}^{(\alpha)}),\nonumber
\end{align}
where different $\alpha$ yield  $\mathbf{x}^{(\alpha)}$ and $\mathbf{y}^{(\alpha)}$ of varying size.

Conversely, with $\size(\mathbf{u}_i)=(u_i,v_i)$ and cumulative sum $\underline{u}_i:=u_1+\dots+u_i$, 
if 
\begin{align*}\mathbf{U}\mathbf{a}\mathbf{V}^\top
&=\bigstar_{1\leq s\leq \len(\mathbf{a})}
\begin{bmatrix}\e{\mu_{(\underline{u}_{(s-1)}+1)}}&\cdots&\e{\mu_{\underline{u}_{s}}}\end{bmatrix}\mathbf{v}_s
\begin{bmatrix}\e{\iota_{(\underline{v}_s+1)}}&\cdots&\e{\iota_{\underline{v}_{(s+1)}}}\end{bmatrix}^\top\\
&=\diag(\mathbf{x},\mathbf{y})
\end{align*}
with $\size(\mathbf{x})=(j_1,k_1)$ then all $s$-indexed factors 
\[\begin{bmatrix}\e{\mu_{(\underline{u}_{(s-1)}+1)}}&\cdots&\e{\mu_{\underline{u}_{s}}}\end{bmatrix}\mathbf{v}_s
\begin{bmatrix}\e{\iota_{(\underline{v}_s+1)}}&\cdots&\e{\iota_{\underline{v}_{(s+1)}}}\end{bmatrix}^\top
=:\diag(\mathbf{x}_1^{(s)},\mathbf{y}_1^{(s)})\]
decompose with $\size(\mathbf{x}_1^{(s)})=(j_1,k_1)$.
Furthermore,  ${\mathbf{x}}_1^{(s)}$ has all entries equal
to $\varepsilon$ or ${\mathbf{y}}_1^{(s)}$ does.
With $\mathbf{U}\mathbf{a}\mathbf{V}^\top\in\composition$ and increasing $\mu$ and $\iota$, there is a unique  $0\leq\alpha\leq\len(\mathbf{a})$ such that $\mathbf{x}_1^{(1)},\ldots,\mathbf{x}_1^{(\alpha)}$ have entries different from $\varepsilon$, and $\mathbf{x}_1^{(\alpha+1)},\ldots,\mathbf{x}_1^{(\len(\mathbf{a}))}$ have all entries equal to $\varepsilon$.
Therefore, $\mathbf{U}$ and $\mathbf{V}$ are of shape (\ref{eq:def_alpha_dec}). 
If $\mathbf{U}\mathbf{a}\mathbf{V}^\top=\diag(\mathbf{x}',\mathbf{y}')$ has a second decomposition with resulting splitting $\alpha'\in{\mathfrak{S}}_{\mathbf{U},\mathbf{V}}^{\mathbf{a}}$ then $\rows(\mathbf{x})<\rows(\mathbf{x}')$ and $\cols(\mathbf{x})<\cols(\mathbf{x}')$  without loss of generality, thus $\alpha<\alpha'$. 
\end{proof}
\index[general]{Delta@$\widetilde\Delta$}

\begin{proof}[Proof of \Cref{lem:Phi_hom_coprod}]
Let $\mathbf{a}=\diag({\mathbf{u}}_1,\dots,{\mathbf{u}}_a)$ with $\len(\mathbf{a})=a$, and for simplicity,  \[\widetilde\Delta(\mathbf{b}):=\Delta(\mathbf{b})-\mathbf{b}\otimes\ec-\ec\otimes\mathbf{b}\] for every $\mathbf{b}$. Then, with \Cref{lem:splittings}, 
\begin{align*}
\widetilde\Delta\circ\Phi(\mathbf{a})
&=\sum_{u,v}\sum_{\substack{\mathbf{U}\in \brew(u,\rows(\mathbf{a}))\\\mathbf{V}\in \brew(v,\cols(\mathbf{a}))}}\frac{1}{\mathbf{U}!\mathbf{V}!}\,\Delta(\mathbf{U}\mathbf{a}\mathbf{V}^\top)-\Phi(\mathbf{a})\otimes\ec-\ec\otimes\Phi(\mathbf{a})\\
&=\sum_{u,v}\sum_{\substack{\mathbf{U}\in \brew(u,\rows(\mathbf{a}))\\\mathbf{V}\in \brew(v,\cols(\mathbf{a}))}}\frac{1}{\mathbf{U}!\mathbf{V}!}\sum_{\substack{\mathbf{x},\mathbf{y}\in\composition\setminus\{\ec\}\\\diag(\mathbf{x},\mathbf{y})=\mathbf{UaV}^\top}}\,\mathbf{x}\otimes\mathbf{y}\\
&\overset{(\ref{eq:UaV_alpha_decomp_diag_form})}{=}
\sum_{u,v}\sum_{\mathbf{U},\mathbf{V}}\frac{1}{\mathbf{U}!\mathbf{V}!}\sum_{\alpha\in\mathfrak{S}_{\mathbf{U},\mathbf{V}}^{\mathbf{a}}}
\mathbf{U}_1\diag(\mathbf{u}_1,\dots,\mathbf{u}_\alpha)\mathbf{V}_{1}^\top
\otimes
\mathbf{U}_{2}\diag(\mathbf{u}_{\alpha+1},\dots,\mathbf{u}_a)\mathbf{V}_{2}^\top\\
&\overset{\ref{lem:factorial_of_brew_multiplic}}{=}\sum_{u,v}\sum_{\mathbf{U},\mathbf{V}}\sum_{\substack{1\leq\alpha<a\\\text{such that }\mathbf{U},\mathbf{V}\\
\alpha\text{-decomposable}}}\frac{\mathbf{U}_{1}\diag(\mathbf{u}_1,\dots,\mathbf{u}_\alpha)\mathbf{V}_{1}^\top}
{\mathbf{U}_{1}!\mathbf{V}_{1}!}
\otimes
\frac{\mathbf{U}_{2}\diag(\mathbf{u}_{\alpha+1},\dots,\mathbf{u}_\alpha)\mathbf{V}_{2}^\top}
{\mathbf{U}_{2}!\mathbf{V}_{2}!}\\
&=\sum_{1\leq\alpha<a}\sum_{u,v}\sum_{\substack{\mathbf{U},\mathbf{V}\\\alpha\text{-decomposable}}}
\frac{\mathbf{U}_{1}\diag(\mathbf{u}_{1},\dots,\mathbf{u}_\alpha)\mathbf{V}_{1}^\top}
{\mathbf{U}_{1}!\mathbf{V}_{1}!}
\otimes
\frac{\mathbf{U}_{2}\diag(\mathbf{u}_{\alpha+1},\dots,\mathbf{u}_a)\mathbf{V}_{2}^\top}
{\mathbf{U}_{2}!\mathbf{V}_{2}!}\\
&\overset{\ref{lem:diag_of_brew_multiplic}}{=}\sum_{1\leq\alpha<a}\left(\sum_{\substack{u_1\\v_1}}\sum_{\substack{\mathbf{U}_1\\\mathbf{V}_1}}
\frac{\mathbf{U}_{1}\diag(\mathbf{u}_{1},\dots,\mathbf{u}_\alpha)\mathbf{V}_{1}^\top}
{\mathbf{U}_{1}!\mathbf{V}_{1}!}\right)
\otimes
\left(
\sum_{\substack{u_2\\v_2}}\sum_{\substack{\mathbf{U}_2\\\mathbf{V}_2}}
\frac{\mathbf{U}_{2}\diag(\mathbf{u}_{\alpha+1},\dots,\mathbf{u}_a)\mathbf{V}_{2}^\top}
{\mathbf{U}_{2}!\mathbf{V}_{2}!}\right)
\\
&=\sum_{1\leq\alpha<a}\Phi(\diag(\mathbf{u}_1,\dots,\mathbf{u}_\alpha))\otimes\Phi(\diag(\mathbf{u}_{\alpha+1},\dots,\mathbf{u}_a))\\
&=(\Phi\otimes\Phi)\circ\widetilde\Delta(\mathbf{a})
\end{align*}
and thus the claim with $\Phi(\ec)=\ec$. 
\end{proof}

Putting it all together, we can prove the main result of this section. 

\begin{proof}[Proof of \Cref{thm_iso_sh_qsh}.]
     Use \Cref{def:Phi} and \Cref{lem:Phi_hom,lem:Phi_inv,lem:Phi_hom_coprod}.
\end{proof}

\section{Implicit isomorphisms of Hopf algebras}
\label{sec:Implicit}
Let us recall from \Cref{rem:identificationdiagAndStringCon} that as a coalgebra, \(H=\groundRing\langle \compositionConnected\rangle\) is isomorphic to a deconcatenation coalgebra, and therefore it is cofree over \(\compositionConnected\).
This means in particular that we have some tools at our disposal for defining Hopf maps to and from \(H\).
We denote by \(\pi:H\to\compositionConnected\) the projection, orthogonal to
\[
  \bigoplus_{\mathbf{a}\in\composition\setminus\compositionConnected}\groundRing\mathbf{a}.
\]
We note that for any \(\mathbf{a},\mathbf{b}\in\composition\),
\begin{equation}\label{eq:pi.shf.zero}
  \pi(\mathbf{a}\blockShuffle\mathbf{b})=0.
\end{equation}
\index[general]{pi@$\pi$}
Let \((C,\Delta_C)\) be any graded connected coalgebra.
By cofreeness of $H$,
a coalgebra homomorphism \(\Psi: C\to H\) is characterized by its projectiol 
$$\psi\coloneq\pi\circ\Psi\colon C^+\to\compositionConnected,$$
where \(C^+=\ker\epsilon_C\) denotes the augmentation ideal, and moreover
\[
  \Psi = \sum_{n\ge 1}\mathrm{diag}\circ\psi^{\otimes n}\circ\widetilde{\Delta}_C^{(n-1)}.
\]
Further still, \(\Psi\) is a coalgebra isomorphism if and only if \(\psi\) is a linear isomorphism.

As a consequence, two coalgebra morphisms \(\Psi,\Psi': C\to H\) coincide if and only if \(\psi=\psi'\).
Applying this to \(C=H\otimes H\), and the maps \(\blockShuffle\circ(\Psi\otimes\Psi)\) and \(\Psi\circ\shuffle\), we see that a coalgebra morphism \(\Psi\) is a bialgebra morphism, if and only if these are equal. 
By cofreeness and \cref{eq:pi.shf.zero} this condition is equivalent to requiring that for any \(\mathbf{a},\mathbf{b}\in\composition\),
\begin{equation}\label{eq:bialg.cond}
  \psi(\mathbf{a}\shuffle\mathbf{b})=0.
\end{equation}

One can make use of this condition together with certain maps that decompose elements in \(H\) as two-parameter shuffle polynomials to obtain bialgebra (hence Hopf algebra) homomorphisms from \((H,\shuffle)\) to \((H,\blockShuffle)\).
\subsection{Leray's theorem}
\index[general]{starcon@$*$}
\index[general]{euler@$\mathfrak{e}$}
\index[general]{mu@$\mu$}
\index[general]{epsilon@$\epsilon$}
In this section \(m\) denotes either \(\shuffle\), \(\qShuffle\) or \(\blockShuffle\).
For $f,g\in\End(H)$ we have the convolution product 
\[
  f*g := m\circ(f\otimes g)\circ\Delta\in\End(H),
\]
and in particular, $f^{*0}:=\id$ and $f^{*k}:=f*f^{*(k-1)}$ for every $k\in\N$. The mapping $\id-\mu\circ\epsilon$ with counit $\epsilon$ and unit $\mu$ is \DEF{$*$-nilpotent}, i.e., for all $a\in H$ there is $k\in\N$ such that $f^{*k}(a)=0$.
This defines the \DEF{Eulerian idempotent}
\[\mathfrak{e}:=\log^*(\id):=\sum_{k\geq 1}\frac{{(-1)}^{k-1}}{k}{(\id-\mu\circ\epsilon)}^k\in\End(H).\]

Note that \[{(\ker\epsilon)}^2:=\{m(a\otimes b)\mid a,b\in\ker\epsilon\}\]
is a $\groundRing$-submodule of $\ker\epsilon$. 

For the following two results see \cite[Sec.~7]{milnor1965structure}, \cite[Sec.~4.3] {cartier2021classical} or \cite[Thm.~1.7.29]{grinberg2020hopf}.
\begin{lemma}\label{lem:euler}
We have 
    \begin{enumerate}
        \item $\ker\mathfrak{e}=\groundRing+{(\ker\epsilon)}^2$ and  
        \item  $\Ima\mathfrak{e}\cong\faktor{(\ker\epsilon)}{{(\ker\epsilon)}^2}$ as $\groundRing$-modules. 
    \end{enumerate}
\end{lemma}

\begin{theorem}[Leray]\label{thm:leray}
    $H\cong \Sym(\Ima\mathfrak{e})$ as $\groundRing$-algebras. 
\end{theorem}

If $\groundRing$ is a field and $B$ a $\groundRing$-basis of $\Ima\mathfrak{e}$, then 
$H\cong\groundRing[B]$ as $\groundRing$-algebras, where $\groundRing[B]$ denotes the polynomial algebra over $B$ considered as a set of symbols. 

By definition the identity
\[
  \mathrm{id}=\exp^*(\mathfrak{e})
\]
holds, thus we get an explicit decomposition of any matrix composition as a polynomial in elements of \(\Ima\mathfrak{e}\),
namely
\[
  \mathbf{u}_1\dots\mathbf{u}_\ell=\sum_{I\in\mathcal{C}(\ell)}\frac{1}{\len(I)!}m\left(\mathfrak{e}(\mathbf{u}_1\dots\mathbf{u}_{I_1})\otimes\dots\otimes\mathfrak{e}(\mathbf{u}_{I_1+\dotsb+I_{\len(I)-1}+1}\dots\mathbf{u}_\ell)\right).
\]

From \Cref{lem:euler} we see that \(\mathfrak{e}\) fulfills condition \eqref{eq:bialg.cond}.
However, it is not always true that \(\Ima\mathfrak{e}\subset\compositionConnected\), as the following example shows.
\begin{example}
Let \(m\) be the two-parameter shuffle product \(\shuffle\). Then
 
  \begin{align*}
    \mathfrak{e}\left( \begin{bmatrix}\w{1}&\varepsilon\\\varepsilon&\w{2}\end{bmatrix} \right) &= \begin{bmatrix}\w{1}&\varepsilon\\\varepsilon&\w{2}\end{bmatrix}-\frac12\left( \begin{bmatrix}\w{1}&\varepsilon\\\varepsilon&\w{2}\end{bmatrix}+\begin{bmatrix}\varepsilon&\w{1}\\\w{2}&\varepsilon\end{bmatrix}+\begin{bmatrix}\varepsilon&\w{2}\\\w{1}&\varepsilon\end{bmatrix}+\begin{bmatrix}\w{2}&\varepsilon\\\varepsilon&\w{1}\end{bmatrix} \right)  \\
    &= \frac12\left(\begin{bmatrix}\w{1}&\varepsilon\\\varepsilon&\w{2}\end{bmatrix}-\begin{bmatrix}\w{2}&\varepsilon\\\varepsilon&\w{1}\end{bmatrix}\right)-\frac12\begin{bmatrix}\varepsilon&\w{1}\\\w{2}&\varepsilon\end{bmatrix}-\frac12\begin{bmatrix}\varepsilon&\w{2}\\\w{1}&\varepsilon\end{bmatrix}\not\in\compositionConnected.
  \end{align*}
\end{example}

Combining \Cref{thm:leray} with results from the previous section we obtain the following result.
\begin{theorem}\label{trm:isoMagic}
  Let \(\Log_\shuffle^\blockShuffle: (H,\shuffle)\to (H,\blockShuffle)\) be the unique bialgebra map induced by \(\pi\circ\mathfrak{e}\).
  It is an isomorphism, and its inverse 
  $$\Exp_\blockShuffle^\shuffle:(H,\blockShuffle)\to(H,\shuffle)$$ 
  is recursively given by
  \[
    \pi_n\left(\Exp_\blockShuffle^\shuffle(\pi_m(\mathbf{a}))\right) = -\sum_{k=n}^m\pi_n\circ{\Exp_\blockShuffle^\shuffle}\circ\pi_k\circ{\Log_\shuffle^\blockShuffle}\circ\pi_m(\mathbf{a}),\quad n\le m,
  \]
  where for \(k\in\mathbb{N}\), \(\pi_k: H\to\compositionConnected^k\) denotes the canonical projection onto the space of matrix compositions with exactly \(k\) connected blocks.
  The same holds after replacing \(\shuffle\) by \(\qShuffle\).
\end{theorem}
\begin{proof}
  Note that the restriction of \(\Log_\shuffle^\blockShuffle\) to \(\compositionConnected\) is the identity map, thus it is a linear isomorphism.
  Hence, it is an isomorphism of algebras.
  The formula for \(\Exp_\blockShuffle^\shuffle\) follows from a general expression for inversion of lower-triangular maps.
\end{proof}
\begin{remark}
    Note that in the terminology of \cite[Prop. 4.1]{Hoffman_2017} the map \(\pi\circ\mathfrak{e}\) is a contraction, and the map \(Log_\shuffle^\blockShuffle\) is the corresponding expansion.
\end{remark}
In particular, by considering appropriate compositions of the exponential and logarithm maps we obtain:
\begin{corollary}\label{thm:isos_Nikolas_Magic}
    We have $(H,\blockShuffle,\Delta)\cong(H,\shuffle,\Delta)\cong(H,\qShuffle,\Delta)$ as Hopf algebras. 
\end{corollary}

\subsection{Application: New proof of the two-parameter bialgebra relation}\label{subsec:applic}

In this section we present a new proof of the bialgebra relation in $(H,\qShuffle,\Delta)$, using the isomorphism to the (classical, one-parameter) block shuffle algebra introduced in this article. 

\begin{proof}[{Proof of \cite[Thm.~5.24]{DS23}}.]
    With isomporphism $\Psi$ from $(H,\blockShuffle,\Delta)$ to $(H,\qShuffle,\Delta)$ according to \Cref{trm:isoMagic}, 
    \begin{align*}
        \Delta\circ\qShuffle
        &=\Delta\circ\qShuffle\circ(\Psi\otimes\Psi)\circ(\Psi^{-1}\otimes\Psi^{-1})\\
        &=\Delta\circ\Psi\circ\blockShuffle\circ(\Psi^{-1}\otimes\Psi^{-1})\\
        &=(\Psi\otimes\Psi)\circ\Delta\circ\blockShuffle\circ(\Psi^{-1}\otimes\Psi^{-1})\\
        &\overset{\ref{thm:1paramHopfAlg}}=(\Psi\otimes\Psi)\circ(\blockShuffle\otimes\blockShuffle)\circ(\id\otimes \tau\otimes\id)\circ(\Delta\otimes\Delta)\circ(\Psi^{-1}\otimes\Psi^{-1})\\
        &=(\qShuffle\otimes\qShuffle)\circ\Psi^{\otimes4}\circ(\id\otimes \,\tau\otimes\id)\circ(\Delta\otimes\Delta)\circ(\Psi^{-1}\otimes\Psi^{-1})\\
        &=(\qShuffle\otimes\qShuffle)\circ(\id\otimes \,\tau\otimes\id)\circ\Psi^{\otimes4}\circ(\Delta\otimes\Delta)\circ(\Psi^{-1}\otimes\Psi^{-1})\\
        &=(\qShuffle\otimes\qShuffle)\circ(\id\otimes \,\tau\otimes\id)\circ(\Delta\otimes\Delta)\circ(\Psi\otimes\Phi)\circ(\Psi^{-1}\otimes\Psi^{-1})\\
        &=(\qShuffle\otimes\qShuffle)\circ(\id\otimes \,\tau\otimes\id)\circ(\Delta\otimes\Delta),
    \end{align*}
    where 
    $\tau$ denotes the flip homomorphism determined by $\tau(\mathbf{a}\otimes\mathbf{b}):=\mathbf{b}\otimes\mathbf{a}$ for all $\mathbf{a},\mathbf{b}\in\composition$. 
\end{proof}

\section{Conclusions and outlook}\label{sec:conclusions}

We have used free Lyndon generators, ideas from $\mathbf{B}_\infty$-algebras, and a novel two-parameter Hoffman's exponential to provide three different classes of isomorphisms between the two-parameter shuffle and quasi-shuffle algebra. In particular, we have provided a Hopf algebraic connection to the (classical, one-parameter) shuffle algebra over the extended alphabet of connected matrix compositions.

\begin{enumerate}
    \item The here suggested machinery should be useful also for \cite{GLNO2022,ZLT22} since it provides free generators and an underlying Hopf algebra structure for two-parameter mapping space signatures. In particular, it is now possible to manage all coefficients that are algebraically dependent through shuffling. This is one of the primary reasons to consider log signatures in the one-parameter setting.
    \item  
    Already in \cite[Sec.~6]{DS23} we addressed the issue to generalize our work for $p>2$ parameters. 
    The new strategy applied in \Cref{subsec:applic} should allow  us to  work around elementary block-tensor arithmetic and prevent us from notational clutter.
    \item 
    The unsolved question from \cite[Sec.~6]{DS23} regarding the computational complexity of the two-parameter sums signature, taken at an arbitrary matrix composition, can now be reduced to a set of free generators introduced in this article. For example \Cref{thm:iso} provides the relation
    $$\begin{bmatrix}
           \w{1}&\varepsilon\\
           \varepsilon&\w{1}
       \end{bmatrix} = \frac12\begin{bmatrix}
           \w{1}
       \end{bmatrix}\qShuffle\begin{bmatrix}
           \w{1}
       \end{bmatrix} -
       \begin{bmatrix}
           \varepsilon&\w{1}\\
           \w{1}&\varepsilon
       \end{bmatrix}
       -
       \begin{bmatrix}
           \w{1}\\
           \w{1}
       \end{bmatrix}
        -
       \begin{bmatrix}
           \w{1}&
           \w{1}
       \end{bmatrix}
       -
       \frac12\begin{bmatrix}
           \w{1}\star
           \w{1}
       \end{bmatrix}
       $$
    which allows us to compute for every eventually zero $Z:\N^2\rightarrow\groundRing^d$ the matrix
    $${\left(\langle\SS_{0_2;t}(Z),\begin{bmatrix}
           \varepsilon&\w{1}\\
           \w{1}&\varepsilon
       \end{bmatrix}\rangle\right)}_{\substack{1\leq t_1\leq T_1\\1\leq t_2\leq T_2}}\in\groundRing^{T_1\times T_2}$$
     in linear time $\mathcal{O}(T_1T_2)$. 
     So far, the (iteratively) chained coefficients from \cite[Thm~4.5]{DS23} and the algebraic relations according  to quasi-shuffles do not yield linear complexity algorithms for all coefficients of the two-parameter signature.  
\emph{Can we understand which coefficients are not computable in linear time via chaining or even other approaches?}
    \item The isomorphisms due to Sections \ref{sec:Lyndongens}, \ref{sec:hoffman} and \ref{sec:Implicit} are all different from each other.  
    While the Lyndon isomorphisms of algebras  (\Cref{cor:isos_lynon}) do not respect the underlying copruduct $\Delta$ dual to diagonal concatenation, the way we define the isomorphism of Hopf algebras $\Phi$ inspired by Hoffman (\Cref{def:Phi}) seems independent of the explicit $\Delta$. \emph{Does $\Phi$ respect other coproducts such as deconcatenation along the antidiagonal? Can we make the isomorhism $\Phi$ unique up to certain conditions? Can we find a natural isomorphism from one-parameter to two-parameter shuffles which does not depend on $\Delta$ using theory related to cumulants?} 
    
\end{enumerate}

\subsection*{Acknowledgements}

LS has been supported by the trilateral French-Japanese-German research project \emph{EnhanceD Data stream Analysis with the Signature Method (EDDA)} of the French National Research Agency (ANR), together with the Japan Science and Technology Agency (JST), and the Deutsche Forschungsgemeinschaft (DFG).
NT is funded by the Deutsche Forschungsgemeinschaft (DFG, German Research Foundation) under Germany's Excellence Strategy – The Berlin Mathematics Research Center MATH+ (EXC-2046/1, project ID: 390685689).

Many of the initial discussions took place during the '23 workshop \emph{Structural Aspects of Signatures and Rough Paths} (as part of the project \emph{Signatures for Images}) at the Center of Advanced Study (CAS) in Oslo, Norway.

 Furthermore, LS would like to thank the \emph{Stochastic Algorithms and Nonparametric Statistics} research group at the WIAS Berlin for the invitation and hospitality from October '23 until March '24.

\newpage

\providecommand{\bysame}{\leavevmode\hbox to3em{\hrulefill}\thinspace}
\providecommand{\MR}{\relax\ifhmode\unskip\space\fi MR }
\providecommand{\MRhref}[2]{%
  \href{http://www.ams.org/mathscinet-getitem?mr=#1}{#2}
}
\providecommand{\href}[2]{#2}

\newpage

\printindex[general]

\end{document}